\documentclass[reqno, 11pt]{amsart}
\usepackage{times}
\usepackage[colorlinks]{hyperref}
\usepackage{url}
\newtheorem{theorem}{Theorem}[section]
\newtheorem{lemma}[theorem]{Lemma}
\newtheorem{proposition}[theorem]{Proposition}
\newtheorem{corollary}[theorem]{Corollary}

\theoremstyle{definition}
\newtheorem{definition}[theorem]{Definition}
\newtheorem{ex}[theorem]{Example}

\newtheorem{remark}[theorem]{Remark}

\theoremstyle{remark}

\numberwithin{equation}{section}

\usepackage{bbm}
\usepackage{url}
\usepackage{euscript}
\usepackage{sec}
\usepackage{pb-diagram}
\usepackage{lamsarrow}
\usepackage{pb-lams}
\usepackage{amsmath}
\usepackage{amsthm}

\usepackage{graphicx}
\usepackage{epstopdf}

\usepackage{amssymb}
\usepackage{pifont}
\usepackage{amsbsy}

\newskip\aline \newskip\halfaline
\def\skipaline{\vskip\aline}

\def\qedbox{$\rlap{$\sqcap$}\sqcup$}
\def\qed{\nobreak\hfill\penalty250 \hbox{}\nobreak\hfill\qedbox\skipaline}

\def\proofend{\eqno{\mbox{\qedbox}}}

\newcommand{\one}{\mathbbm{1}}

\newcommand{\uc}{\underline{\mathbb C}}

\newcommand\bC{{\mathbb C}}

\newcommand\bR{{\mathbb R}}
\newcommand\bS{{\mathbb S}}

\newcommand\bZ{{\mathbb Z}}

\DeclareMathOperator{\im}{\mathbf{Im}}

\DeclareMathOperator{\tr}{{\rm tr}}

\DeclareMathOperator{\sign}{sign}

 \DeclareMathOperator{\res}{{\bf
Res}}

\DeclareMathOperator{\codim}{codim}

\DeclareMathOperator{\diag}{Diag} 
 \DeclareMathOperator{\End}{End}

\newcommand{\bsA}{\boldsymbol{A}}
\newcommand{\ba}{{\boldsymbol{a}}}
\newcommand{\bc}{\boldsymbol{c}}
\newcommand{\hbc}{\hat{\boldsymbol{c}}}

\newcommand{\be}{\boldsymbol{e}}
\newcommand{\hbe}{\hat{\boldsymbol{e}}}

\newcommand{\ii}{\boldsymbol{i}}

\newcommand{\bu}{\boldsymbol{u}}

\newcommand{\bsE}{\boldsymbol{E}}
\newcommand{\bsH}{\boldsymbol{H}}
\newcommand{\bsI}{\boldsymbol{I}}

\newcommand{\bsT}{{\boldsymbol{T}}}

\newcommand{\bthet}{\boldsymbol{\theta}}

\newcommand{\bnu}{\boldsymbol{\nu}}

\newcommand{\bXi}{\boldsymbol{\Xi}}

\newcommand{\bom}{\boldsymbol{\omega}}

\newcommand{\si}{{\sigma}}

\newcommand{\ve}{{\varepsilon}}
\newcommand{\eps}{{\epsilon}}
\newcommand{\vfi}{{\varphi}}

\newcommand{\eA}{\EuScript{A}}
\newcommand{\eB}{\EuScript{B}}

\newcommand{\eD}{\EuScript{D}}

\newcommand{\eH}{\EuScript H}
\newcommand{\eI}{\EuScript{I}}
\newcommand{\eJ}{\EuScript{J}}
\newcommand{\eK}{\EuScript{K}}
\newcommand{\eL}{\EuScript{L}}

\newcommand{\eO}{\EuScript{O}}
\newcommand{\eP}{{\EuScript{P}}}

\newcommand{\eR}{\EuScript{R}}
\newcommand{\eS}{\EuScript{S}}

\newcommand{\eU}{\EuScript{U}}

\newcommand{\ra}{\rightarrow}
\newcommand{\hra}{\hookrightarrow}
\newcommand{\Lra}{{\longrightarrow}}

\def\inpr{\mathbin{\hbox to 6pt{\vrule height0.4pt width5pt depth0pt \kern-.4pt \vrule height6pt width0.4pt depth0pt\hss}}}

\newcommand{\pa}{\partial}
\newcommand{\bpar}{\bar{\partial}}

\newcommand{\nah}{\widehat{\nabla}}

\newcommand{\hg}{\hat{g}}
\newcommand{\wbe}{\widehat{\boldsymbol{E}}}

\DeclareMathOperator{\Dom}{Dom}

\newcommand{\Lag}{\mathrm{Lag}}

\DeclareMathOperator{\spec}{spec}
\DeclareMathOperator{\Iso}{Iso}

\begin{document}

\title{Dirac  operators  on cobordisms: degenerations and surgery}
\date{Started July 14, 2009. Completed  on August 14, 2009. This is the   \emph {\today} version.}

\author{Daniel F. Cibotaru}

\address{Department of Mathematics, University of Notre Dame, Notre Dame, IN 46556-4618.}
\email{cibotaru.1@nd.edu}

\author{Liviu I. Nicolaescu}

\address{Department of Mathematics, University of Notre Dame, Notre Dame, IN 46556-4618.}
\email{nicolaescu.1@nd.edu}



\begin{abstract} We investigate the Dolbeault operator on a pair of pants,  i.e., an elementary cobordism between a circle and the disjoint union of two circles.  This  operator  induces  a canonical selfadjoint   Dirac operator   $D_t$  on each regular level set $C_t$ of a fixed  Morse function  defining  this cobordism. We show that as we approach the critical level set  $C_0$ from above and from below these  operators    converge  in the gap topology to  (different) selfadjoint  operators $D_\pm$ that we describe explicitly. We also relate the   Atiyah-Patodi-Singer index of the Dolbeault operator on  the cobordism  to   the spectral flows  of the operators $D_t$ on the complement of $C_0$ and   the Kashiwara-Wall  index of a triplet of  finite dimensional lagrangian spaces canonically determined by $C_0$.
\end{abstract}

\maketitle

\tableofcontents

\section*{Introduction}
\addtocontents{section}{Introduction}
Suppose  $(M, g)$ is  compact oriented   odd dimensional Riemann manifold. We let  $\widehat{M}$denote  the cylinder  $[0,1]\times M$ and $\hat{g}$ denote the cylindrical metric $dt^2+g$.

Let $\hat{D}$ be  a first order elliptic operator operator  on $\widehat{M}$   that has the form
\begin{equation*}
\widehat{D}=\si(dt)\bigl(\, \nabla_t -  D(t)\,\bigr),
\tag{$\dag$}
\label{tag: dag}
\end{equation*}
where $\si$ denotes the principal symbol of $\widehat{D}$, and for every  $t\in [0,1]$ the operator $D(t)$ on $\{t\}\times M$ is elliptic and symmetric. For simplicity we assume that both $A(0)$ and $A(1)$ are  invertible.

A classical result of Atiyah, Patodi and Singer \cite[\S 7]{APS3} (see also \cite[\S 17.1]{KM})  relates the index $i_{APS}(\widehat{A})$ of the   Atiyah-Patodi-Singer problem  associated to $\widehat{D}$ to the spectral flow $SF(\, D(t)\,)$ of the family of Fredholm selfadjoint operators   $D(t)$. More precisely, they show that
\begin{equation*}
i_{APS}(\widehat{D})+ SF\bigl(\, D(t),\;0\leq t\leq 1\,\bigr)=0.
\tag{A}
\label{tag: a}
\end{equation*}
We can regard  the cylinder $\widehat{M}$ as a trivial     cobordism   between $\{0\}\times M$ and $\{1\}\times M$, and the coordinate $t$ as a Morse function on $\widehat{M}$   with no critical points.

In this paper   we initiate  an investigation   of the case   when $\widehat{M}$ is no longer a trivial  cobordism.  We outline below the   main themes of this investigation.

First,   we will  concentrate only on elementary cobordisms,   the ones that trace a single surgery.  We regard  such a cobordism as a pair $(\widehat{M}, f)$,  where $\widehat{M}$ is an even dimensional, compact oriented    manifold  with boundary, and $f$ is a Morse function on $\widehat{M}$  with a single critical point $p_0$ such that
\[
f(\widehat{M})=[-1,1],\;\; f(\pa M)=\{-1,1\},\;\;f(p_0)=0.
\]
We set $M_{\pm }:=f^{-1}(\pm 1)$ so that  we have a diffeomorphism of oriented manifolds  $\pa M= M_+\cup -M_{-}$. By removing the critical level set $M_0=f^{-1}(0)$ we obtain two cylinders 
\[
\widehat{M}_-=\{f<0\}\cong [-1, 0)\times M_{-},\;\;\widehat{M}_+=\{f>0\}\cong (0,1]\times M_+.
\]
Suppose $\hat{g}$ is a Riemann metric on $\widehat{M}$ and $\widehat{D}: C^\infty(E_+)\ra C^\infty (E_-)$ is a Dirac type operator  on $\widehat{M}$, where $E_+\oplus E_-$ is a $\bZ/2$-graded  bundle of Clifford modules.

 Using the bundle isomorphism  $\si(dt)$ we can regard $\widehat{D}$ as an operator $C^\infty(E_+)\ra C^\infty(E_+)$. As explained   in \cite{Gil}  (see also Section \ref{s: 2} of this paper), for every $t\neq 0$,  there is a canonically induced  symmetric Dirac operator $D(t)$ on the slice $M_t=f^{-1}(t)$.   We regard $D(t)$ as a linear operator $D(t): C^\infty(E_+|_{M_t})\ra   C^\infty(E_+|_{M_t})$, so that  if  $\hat{g}$ were a cylindrical metric then   formula (\ref{tag: dag}) would hold.    
 
 The Riemann metric $\hat{g}$ defines finite measures $dV_t$ on  all the slices   $M_t$, including the singular slice $M_0$.  In particular we obtain a  one parameter       family of Hilbert spaces 
 \[
 \bsH_t:= L^2(M_t, dV_t; E_+).
 \]
 We can now regard $D(t)$ as a  closed, densely defined linear operator on $\bsH_t$.
 
 \medskip
 
 \noindent {\bf Problem 1.} Organize   the  family $(\bsH_t)_{t\in [-1,1]}$ as a trivial  Hilbert bundle  over the interval $[-1,1]$
 \[
   \eH=\bsH\times [-1,1]\ra [-1,1].
 \]
 
 \smallskip

 Under reasonable  assumptions on $f$  and $\hat{g}$ we can use the gradient flow of  $f$ to address this issue. Once  this problem is solved  we can  regard     the operators $D(t)$, $t\neq 0$    as closed densely defined operators on the same Hilbert space $\bsH$.  We can then formulate   our next problem.
 
 \medskip
 
 \noindent {\bf Problem 2.}    Investigate whether  the limits 
 \[
 SF_-:=\lim_{\ve\searrow 0} SF(\, D(t), -1\leq t\leq -\ve),\;\; SF_+:=\lim_{\ve\searrow 0} SF(\, D(t),\; \ve\leq t\leq 1\,).
 \]
 exist and are finite.

 \medskip

 If  Problem 2 has a positive answer  we  are interested in a  version of (\ref{tag: a})  relating  these limits to the Atiyah-Patodi-Singer index of  $\widehat{D}$   in the noncylindrical  formulation of  \cite{Gil, Grubb}.

 \medskip
 
 \noindent {\bf Problem 3.}   Express the quantity
 \begin{equation*}
 \delta:=i_{APS}(\widehat{D})+ SF_-+SF_+
 \tag{B}
 \label{tag: b}
 \end{equation*}
 in terms  of   invariants  of the singular level  set $M_0$.

 \medskip

 The  existence of the limits  in Problem 2  is  a  consequence of a much more refined analytic behavior  of the family  of operators $D(t)$ that we now proceed to  explain. We  set
     \[
     \widehat{\bsH}:= \bsH\oplus \bsH,\;\; {\bsH}_+=\bsH\oplus 0,\;\; {\bsH}_-=0\oplus \bsH,
     \]
 and     denote by $\Lag$ the Grassmannian of   hermitian lagrangian  subspaces $\widehat{\bsH}$. These are  complex subspaces $L\subset \widehat{\bsH}$ satisfying $L^\perp =JL$, where $J: \bsH\oplus \bsH\ra \bsH\oplus \bsH$  is the operator with block decomposition
 \[
 J=\left[
 \begin{array}{cc}
 0 & -1\\
 1 & 0
 \end{array}
 \right].
 \]
 Following \cite{Cib} we denote by $\Lag^-$ the open subset of  $\Lag$ consisting   of  lagrangians $L$ such that  the pair  of subspaces $(L, \widehat{\bsH}_-)$ is a Fredholm pair, i.e.,
 \[
 L+{\bsH}_-\;\;\mbox{is closed and}\;\;\dim L\cap {\bsH}_- <\infty
 \]
 As explained  in \cite{Cib}, the space  $\Lag^-$ equipped with the gap topology of \cite[\S IV.2]{K} is a classifying spaces for the complex $K$-theoretic functor $K^1$.        
 
 To a closed  densely defined operator $T:\Dom(T)\subset \bsH\ra \bsH$ we associate its switched graph
 \[
 \widetilde{\Gamma}_T:=\Bigl\{ (Th,h)\in \widehat{\bsH};\;\;h\in \Dom(T)\,\Bigr\}.
 \]
 Then  $T$ is selfadjoint if and only if   $\widetilde{\Gamma}_T\in \Lag$.  It is also Fredholm if and only if $\widetilde{\Gamma}_T\in \Lag^-$.    We can now formulate a refinement of   Problem 2.
 
 \medskip
 
 \noindent {\bf Problem} $\mathbf{2^*}$. Investigate whether the limits  $\widetilde{\Gamma}_\pm =\lim_{t\searrow 0} \widetilde{\Gamma}_{D(\pm t)}$ exist in the gap topology and, if so, do  they belong to $\Lag^-$.

 \medskip
 
 The gap  convergence of the  switched graphs of operators is equivalent to the convergence in norm as $t\ra 0^\pm$ of the resolvents $R_t=(\ii + D(t)\,)^{-1}$.  To show  that $\widetilde{\Gamma}_\pm\in \Lag^-$ it suffices to show  that  the limits $R_\pm =\lim_{t\ra 0^\pm} R_t$ are  compact  operators.   If in  addition\footnote{The condition    $\widetilde{\Gamma}_\pm\cap \widehat{\bsH}_-=0$ is not really needed, but  it makes our presentation more transparent. In any case, it is generically satisfied.} $\widetilde{\Gamma}_\pm\cap \widehat{\bsH}_-=0$ then the limits in   Problem 2 exist and are finite.

 An even analog of Problem $2^*$  was investigated in \cite{SS}.  The role of the   smooth slices $M_t$   was played there by a $1$-parameter family  of Riemann surfaces degenerating to  a  Riemann surface  with single singularity  of the simplest type, a  node. The authors show that the gap limit of the   graphs of Dolbeault operators on $M_t$   exists  and then   described it explicitly.  
 
 In this paper we solve Problems 1, $2^*$ and 3 in the symplest possible case, when $\widehat{M}$ is an elementary $2$-dimensional cobordism, i.e., a pair of pants (see Figure \ref{fig: cobdeg1})  and $\widehat{D}$  is the Dolbeault operator on the Riemann surface $\widehat{M}$.
 
 We solved  Problem 1 by an ad-hoc intuitive method.       The limits   $\widetilde{\Gamma}_\pm$ in Problem  $2^*$ turned out to be switched graphs of   certain Fredholm-selfadjoint operators $D_\pm$, $\widetilde{\Gamma}_\pm=\widetilde{\Gamma}_{D_\pm}$.  
 
We describe these operators    as realizations of two different boundary value problems associated to the same   symmetric Dirac operator $D_0$  defined  on  the disjoint union of  four intervals. These intervals are obtained by removing the  singular point of the critical level set  $M_0$ and then cutting in two each of the resulting two components. The boundary conditions    defining $D_\pm$ are described by  some ($4$-dimensional)  lagrangians $\Lambda_\pm$  determined by the geometry of the singular slice $M_0$. The operators $D_\pm$  have well defined    eta invariants $\eta_\pm$.  If $\ker D_\pm = 0$ then we can express the defect $\delta$ in (\ref{tag: b}) as
 \begin{equation*}
 \delta= \frac{1}{2}\bigl(\,\eta_--\eta_+\,\bigr).
 \tag{C}
 \label{tag: c}
 \end{equation*}
 The above difference of eta invariants admits a purely symplectic  interpretation very similar to the signature additivity defect of  Wall \cite{Wall}. More precisely, we show that
 \begin{equation*}
 \delta=-\omega(\, \Lambda_0^\perp, \Lambda_+,\Lambda_-\,\bigr),
 \tag{D}
 \label{tag: d}
 \end{equation*}
 where $\Lambda_0$ is the Cauchy data space of the operator $D_0$ and $\omega(L_0,L_1,L_2)$ denotes the Kashiwara-Wall  index of a triplet of lagrangians canonically determined  by $M_0$; see  \cite{CLM, KL, Wall} or Section \ref{s: 4}.
 
 Here is briefly how we structured the paper.  In Section \ref{s: 1}    we investigate  in great detail  the type of degenerations that occur  in the family $D(t)$ as $t\ra 0^\pm$.      It   boils down to understanding  the behavior of families of operators of the  unit circle $S^1$ of the type
 \[
L_\ve= -\ii\frac{d}{d\theta} + a_\ve(\theta),
 \]
 where    $\{a_\ve\}_{\ve >0}$ is a family of smooth  functions on the unit circle that converges   in a rather weak  sense way as $\ve \ra 0$  to a Dirac  measure supported at a point $\theta_0$.    For example if we think of $a_\ve$ as densities   defining measures  converging weakly to the Dirac measure, then the   corresponding family of operators has a well defined gap limit; see Corollary \ref{cor: cont}. 
 
   In  Theorem \ref{th: split} we give an explicit  description of this limiting operator as an  operator   realizing  a natural boundary value   problem on the    \emph{disjoint} union of the two intervals, $[0,\theta_0]$ and $[\theta_0,2\pi]$.   This section also contains a detailed  discussion of the eta invariants of operators of the type  $-\ii\frac{d}{d\theta} + a(\theta)$, where $a$ is a allowed to be the ``density'' of  any finite Radon measure.

In Section 2 we survey mostly known  facts concerning the Atiyah-Patodi-Singer problem when the  metric near the boundary is not cylindrical.   Because the  various orientation conventions vary wildly in the existing literature, we decided to go careful through   the computational details.   We discuss two topics. First, we explain  what is the restriction of a Dirac operator to a cooriented hypersurface   and relate this construction    to  another conceivable  notion of restriction.  In the second part of this section we discuss the noncylindrical version  of the  Atiyah-Patodi-Singer index  theorem. Here we follow closely the presentation in \cite{Gil, Grubb}.

In Section \ref{s: 3} we formulate  and prove the main result of this paper, Theorem \ref{th: main}.      The solution  to Problem $2^*$ is obtained by   reducing the study of the degenerations to the model degenerations  investigated in  Section \ref{s: 1} The equality (\ref{tag: c})  follows  immediately from the noncyclindrical version of the Atiyah-Patodi-Singer index theorem discussed in Section \ref{s: 2} and the eta invariant computations in Section \ref{s: 1}.    In the last  section we  present a few facts about  the Kashiwara-Wall  triple index    and   then use them to prove (\ref{tag: d}).  Our definition of triple index is the one  used  by Kirk and Lesch \cite{KL} that generalizes to infinite dimensions.

Finally a few words about  conventions and notation.    We consistently orient the boundaries using the outer-normal-first convention. We let $\ii$ stand for $\sqrt{-1}$ and we let $L^{k,p}$ denote Sobolev spaces of   functions that  have weak derivatives up to order $k$ that belong to $L^p$.

\section{A model degeneration}
\label{s: 1}
\setcounter{equation}{0}
Let $L>0$ be a positive number. Denote by $\bsH$ the Hilbert space $L^2([0,L],\bC)$.  To any smooth function $a:\bR\ra \bR$ which is $L$-periodic we associate the    selfadjoint operator
\[
D_a:  \Dom (D_a)\subset  \bsH\ra \bsH,
\]
where
\begin{equation}
\Dom(D_a)=\bigl\{ \,u\in L^{1,2}([0,L],\bC);\;\;u(0)=u(L)\,\bigr\},\;\;D_a u=-\ii\frac{du}{dt} +au.
\label{eq: da}
\end{equation}
In this section we would like to  understand  the dependence of $D_a$ on the potential $a$, and in particular, we would like to allow for more singular potentials  such as a Dirac distribution concentrated  at an interior point of the interval. We will reach this goal via a limiting procedure that we implement in  several steps.

We  observe  first  that $D_a$ can be expressed in terms of the resolvent  $R_a:=(\ii+D_a)^{-1}$ as $D_a= R_a^{-1} -\ii$. The advantage of this point of view   is that we can express   $R_a$ in terms   of   the more regular function
\begin{equation*}
A(t):=\int_0^t a(s) ds.
\tag{$\ast$}
\label{tag: ast}
\end{equation*}
which continues to make sense even  when  there is no   integrable function $a$ such that  (\ref{tag: ast}) holds.  For example, we can allow  $A(t)$  to be any function with bounded variation so that,    formally, $a$ ought  to be the density of any Radon measure on $[0,L]$. 

 This will allow us to conclude  that when  we have a family of smooth potentials  $a_n$  that converge  in a  suitable sense to   something singular such as a Dirac function, then the  operators $D_{a_n}$ have a limit in the gap topology to a Fredholm selfadjoint  operator with compact rezolvent.  We show that in many cases this limit operator can be expressed as  the   Fredholm operator defined by a boundary value problem.

We  begin  by expressing  $R_a$ as an integral operator. We set
\[
A(t):=\int_{0}^t a(s) ds,\;\;\Phi_A(t):=\ii A(t)-t.
\]
For $f\in H$ the function $u=R_af$ is the solution of the  boundary value problem
\[
\Bigl(\,\ii-\ii\frac{d}{dt}\,\Bigr)u +au=f,\;\;u(0)=u(L).
\]
We rewrite the above equation as
\[
\frac{du}{dt}+(\ii a-1)u =\ii f
\]
from which we deduce
\[
\frac{d}{dt} \Bigl(\,e^{\Phi_A(t)}u(t)\,\Bigr)= \ii e^{\Phi_A(t)} f(t).
\]
This implies  that
\[
e^{\Phi_A(t)}u(t)-u(0)=\ii \int_0^t  e^{\Phi_A(s)} f(s) ds,\;\;\forall t\in [0,L].
\]
If in the above equality we let $t=L$ and use condition $u(0)=u(L)$  we deduce
\[
u(0)= \frac{\ii}{ e^{\Phi_A(L)}-1}\int_0^L e^{\Phi_A(s)} f(s) ds.
\]
Finally we deduce
\begin{equation}
u(t)=R_a f= \frac{\ii e^{-\Phi_A(t)}}{ e^{\Phi_A(L)}-1}\int_0^L e^{\Phi_A(s)} f(s) ds+ \ii \int_0^t  e^{-(\Phi_A(t)-\Phi_A(s)} f(s) ds.
\label{eq: r1}
\end{equation}
The key point of the above  formula is that $R_a$ can   be expressed  in terms of the antiderivative $A(t)$ which typically has milder singularities than $a$. To analyze the dependence of $R_a$ on $A$  we   introduce a  class of admissible functions.

\begin{definition} (a) We say that $A:[0,L]\ra  \bR$ is \emph{admissible}    if  $A$   has bounded variation, it is right continuous,   and $A(0)=0$. We denote by $\eA$ or $\eA_L$ the class of admissible functions.

\noindent (b)    We say that  a sequence $\{A_n\}_{n\geq 0}\subset \eA$ converges \emph{very weakly} to $A\in \eA$ if   there exists   a negligible   subset $\Delta \subset (0,L)$ such that
\[
\lim_{n\ra \infty} A_n(t)=A(t),\;\;\forall t\in [0,L]\setminus \Delta.\proofend
\]
\label{def: a}
\end{definition}

\begin{remark}   (a)  Note that if $A_n$ converges   very weakly to $A$ then $A_n(L)$ converges to $A(L)$.

\noindent  (b)  Let us explain  the motivation behind  the ``very weak'' terminology.   An admissible   function $A$ defines a finite Lebesgue-Stieltjes   measure $\mu_A$ on $[0,L]$, and the resulting map  $A\mapsto \mu_A$ is a linear isomorphism between $\eA$ and the space of finite Borel measures on $[0,L]$, \cite[Thm. 3.29]{Fol}.  Thus, we can identify $\eA$ with the space of finite Borel  measures on $[0,L]$. As such it is equipped with a weak topology.

According to \cite[\S 4.22]{Ed},  a sequence of  Borel measures $\mu_{A_n}$  is weakly convergent to  $\mu_A$ if and only if $\mu_{A_n}(\eO)\ra \mu_{A}(\eO)$, for any  (relatively) open subset $\eO$ of $[0,L]$.  This clearly  implies  the very  weak  convergence introduced in Definition \ref{def: a}. \qed
\label{rem: weak}
\end{remark}

 Inspired by (\ref{eq: r1}) we define for every $A\in \eA$ the function $\Phi_A(t)= \ii A(t)-t$ and the integral kernels
\[
\eS_A:[0,L]\times [0,L]\ra \bC,\;\; \eS_A(t,s)= \frac{\ii }{ e^{\Phi_A(L)}-1} e^{-\bigl(\,\Phi_A(t)-\Phi_A(s)\,\bigr)},\;\;\forall t,s\in [0,L],
\]
\[
\eK_A:[0,L]\times [0,L]\ra \bC,\;\;\eK_A(t,s)=\begin{cases}
0 &  t<s\\
\ii e^{-\bigl(\,\Phi_A(t)-\Phi_A(s)\,\bigr)}& t\geq s.
\end{cases}
\]
Observe that there exists a constant $C>0$ such that
\begin{equation}
\|\eS_A\|_{L^\infty([0,L]\times [0,L])}+\|\eK_A\|_{L^\infty([0,L]\times [0,L])}\leq C,\;\;\forall A\in \eA.
\label{eq: bounded}
\end{equation}
Thus, these kernels    define   bounded compact operators $S_A,K_A:\bsH\ra \bsH$; see \cite[\S X.2]{Yos}.  Moreover, if we denote by $\|\bullet\|_{\rm op}$  the operator norm  on the space $\eB(\bsH)$ of   bounded linear operators $\bsH\ra \bsH$ then we have the estimates  that
\begin{equation}
\|S_A\|_{\rm op}\leq \|\eS_A\|_{L^2([0,L]\times [0,L])},\;\; \|K_A\|_{\rm op}\leq \|\eK_A\|_{L^2([0,L]\times [0,L])}.
\label{eq: b1}
\end{equation}
We can now rewrite (\ref{eq: r1}) as
\begin{equation}
R_a = R_A:=S_A+ K_A.
\label{eq: r2}
\end{equation}
 \begin{proposition}  If  $A_n$ converges very weakly  to $A$ then $S_{A_n}$ and $K_{A_n}$ converge in the operator norm topology to $S_A$ and respectively $K_A$.
\label{prop: cont}
\end{proposition}

\begin{proof} The very weak convergence  implies that
\[
\eS_{A_{n}}(t,s) \stackrel{k\ra \infty}{\Lra}  \eS_{A}(t,s),\;\;\eK_{A_{n}}(t,s) \stackrel{k\ra \infty}{\Lra}  \eK_{A}(t,s)\;\;\mbox{a.e. on $[0,L]\times [0,L]$}.
\]
Using (\ref{eq: bounded}), the above pointwise convergence and the dominated convergence theorem we deduce
\[
\lim_{n\ra \infty}\Bigl(\,\|\eS_{A_{n}}-\eS_A\|_{L^2([0,L]\times[0,L])}+ \|\eK_{A_{n}}-\eK_A\|_{L^2([0,L]\times[0,L])}\,\Bigr)=0.
\]
Using (\ref{eq: b1}) we deduce that
\[
\lim_{n\ra \infty}\Bigl(\,\|S_{A_{n}}-S_A\|_{\rm op}+ \|S_{A_{n}}-S_A\|_{\rm op}\,\Bigr)=0.\]
\end{proof}
We  want to describe the spectral decompositions of the  operators $R_A$, $A\in\eA$.     To do this we rely on the fact that for certain  $A$'s  the operator  $R_A$ is the resolvent of an elliptic selfadjoint operator on $S^1$. We use this to  produce an intelligent guess for the spectrum of $R_A$ in general.

Let $a$ be a smooth, real valued, $L$-period function on $\bR$ and form again the operator $D_a$ defined in (\ref{eq: da}). We set as usual
\[
A(t)=\int_0^t a(s) ds.
\]
The operator  $D_a$ has discrete real spectrum. If $u(t)$ is an eigenfunction corresponding to an eigenvalue $\lambda$ then
\[
-\ii \frac{du}{dt} +a u=\lambda u \Rightarrow  \frac{du}{dt} +i(a-\lambda)u=0
\]
so that $u(t)=u(0)  e^{-\ii A(t)+\ii\lambda t}$. The periodicity assumption implies $\lambda L-A(L)\in 2\pi\bZ$ so the spectrum of $D_a$ is
\begin{equation}
\spec(D_a)=\left\{\, \lambda_{A,n}:=\frac{2\pi}{L}\bigl(\,\omega_A+n\,\bigr);\;\;n\in \bZ\,\right\},\;\;\mbox{where}\;\; \omega_A:=\frac{A(L)}{2\pi}.
\label{eq: spec}
\end{equation}
The eigenvalue $\lambda_{A,n}$ is simple and the eigenspace   corresponding to $\lambda_{A,n}$ is spanned by
\[
\psi_{A,n}(t):= e^{\frac{2\pi n\ii t}{L}} e^{-\ii (A(t)-\frac{A(L) t}{L})}.
\]
The numbers $\lambda_{A, n}$ and the functions  $\psi_{A,n}$ are well defined \emph{for any} $A\in \eA$.

\begin{lemma} Let $A\in \eA$.  Then  the collection $\{\psi_{A,n}(t);\;\;n\in\bZ\}$  defines a  Hilbert basis of $\bsH$.
\end{lemma}

\begin{proof} Observe  first that the collection
\[
e_n(t)=\psi_{A=0,n}(t)=e^{\frac{2\pi n\ii t}{L}} ,\;\;n\in\bZ
\]
is the canonical Hilbert basis of $\bsH$ that leads to the classical Fourier  decomposition.  The map
\[
U_A:\bsH\ra \bsH,\;\;   \bsH\ni f(t)\mapsto   e^{-\ii (A(t)-\frac{A(L) t}{L})} f(t)
\]
is unitary. It maps  $e_n$ to $\psi_{A,n}$  which  proves our claim.
\end{proof}

A direct computation shows that
\[
R_A\psi_{A,n}=\frac{1}{\ii +\lambda_{A,n}}\psi_{A,n},\;\;\forall A\in \eA,\;\;A\in\eA.
\]
This proves that   for any $A\in \eA$ the collection $\{\psi_{A,n}\}_{n\in\bZ}$ is a    Hilbert basis that diagonalizes the operator $R_A$. Observe that $R_A$ is injective and compact. We define
\[
T_A:= R_A^{-1}-\ii.
\]
The operator $T_A$, is unbounded, closed and densely defined with domain  $\Dom(T_A)= {\rm Range}\,(R_A)$. We will  present later a more  explicit description of $\Dom(T_A)$ for a  large class of $A$'s.  

Note that when
\[
A=\int_0^t a(s) ds,\;\;\mbox{$a$ smooth and $L$-periodic},
\]
the operator $T_A$  coincides with  the operator $D_a$ defined in (\ref{eq: da}).  Proposition \ref{prop: cont} can be rephrased as follows. 

\begin{corollary} If the sequence $(A_n)_{n\geq 1}\subset \eA$ converges very weakly to $A\in \eA$ then the sequence of unbounded operators $(T_{A_n})_{n\geq 1}$ converges in  the gap topology to the unbounded operator $T_A$.\qed
\label{cor: cont}
\end{corollary}

The spectrum of $T_A$ consists only of the simple  eigenvalues $\lambda_{A,n}$, ${n\in\bZ}$.   The     function $\psi_{A_n}$ is an  eigenfunction of $T_A$ corresponding to the eigenvalue $\lambda_{A,n}$.  The eta invariant of $T_A$ is now easy to compute. For $s\in\bC$ we have
\[
\eta_A(s) := \sum_{\lambda>0} \frac{1}{\lambda^s}\Bigl(\dim\ker(\lambda-T_A)-\dim\ker(\lambda+T_A)\,\Bigr)
\]
\[
=\sum_{n\in\bZ\setminus \{-\omega_A\}} \frac{\sign\lambda_{A,n}}{|\lambda_{A,n}|^s}=\frac{L^s}{2\pi^s}\sum_{n\in\bZ\setminus \{-\omega_A\}}\frac{\sign\bigl(\,n+\omega_A\bigr)}{|n+\omega_A|^s}.
\]
Let
\begin{equation}
\rho_A:=\omega_A-\lfloor\omega_A\rfloor=\frac{A(L)}{2\pi}-\left\lfloor\frac{A(L)}{2\pi}\,\right\rfloor\in[0,1).
\label{eq: ro}
\end{equation}
If $\rho_A=0$ then $\eta_A(s)=0$ because in this case the spectrum of $T_A$ is symmetric about the origin. If $\rho_A\neq 0$ then we have
\[
\eta_A(s)=\frac{L^s}{2\pi^s}\Biggl(\,\sum_{n\geq 0} \frac{1}{(n+\rho_A)^s}-\sum_{n\geq 0}\frac{1}{(n+1-\rho_A)^s}\,\Biggr)=\frac{L^s}{2\pi^s}\Bigl(\,\zeta(s,\rho_A)-\zeta(s,1-\rho_A)\,\Bigr),
\]
where for every $a\in (0,1]$  we denoted by $\zeta(s, a)$ the  Riemann-Hurwitz zeta function
\[
\zeta(s,a)=\sum_{n\geq 0}\frac{1}{(n+a)^s}.
\]
The above series is convergent  for any $ s\in \bC$, $\res>1$ and admits an analytic continuation  to the puctured plane $\bC\setminus \{s=1\}$. Its value at the origin $s=0$ is given by Hermite's formula  \cite[13.21]{WW}
\begin{equation}
\zeta(0,a)=\frac{1}{2}-a.
\label{eq: herm}
\end{equation}
We deduce that $\eta_A(s)$ has an analytic continuation at $s=0$ and we have
\begin{equation}
\eta_A(0)=\left\{\begin{array}{ccl}
0 & {\rm if}& \rho_A=0,\\
1-2\rho_A & {\rm if} &\rho_A\in (0,1).
\end{array}
\right.
\label{eq: eta0}
\end{equation}
If we introduce  the  function
\[
\xi_A=\frac{1}{2}\bigl(\,\dim\ker T_A+\eta_A(0)\,\bigr),
\]
then we can rewrite the above equality in a more compact way
\begin{equation}
\xi_A=\frac{1}{2}(1-2\rho_A)=\frac{1}{2}-\rho_A.
\label{eq: xi}
\end{equation}

Suppose  we have $A_0,A_1\in \eA$. We set $A_s=A_0+s(A_1-A_0)\in \eA$.  The map $[0,1]\ni s\mapsto A_s\in \eA$   is continuous in the weak tooplogy on $\eA$ and thus the family of operators  $T_{A_s}$ is  continuous   with respect to the gap topology. The eigenvalues of the family $T_{A_s}$  can be organized in smooth families
\[
\lambda_{s,n}=\frac{2\pi}{L}(\omega_{s}+n)=\frac{2\pi}{L}\Bigl(\omega_{A_0}+s\bigl(\, \omega_{1}-\omega_{0}\,\bigr)+n\,\Bigr),\;\;\omega_s:=\omega_{A_s},;\;\forall s\in [0,1].
\]
Assume for simplicity that $\omega_0,\omega_1\not\in \bZ$, i.e., the operators $T_{A_0}$ and $T_{A_1}$ are invertible. Denote by $SF(A_1, A_0)$ the spectral flow of the affine family\footnote{The quantity $SF(A_1,A_0)$ is independent of the weakly continuous  path $A_s$ connecting $A_0$ to $A_1$ since the space $\eA$ equipped with the   weak topology is contractible. It is thus an invariant of the pair $(A_1,A_0)$.} $T_{A_s}$.   Then
\[
SF(A_1,A_0)=\#\{ n\in\bZ;\;\; \omega_0+n <0,\;\;\omega_1+n>0\}-\# \{ n\in \bZ;\;\;\omega_0+n>0,\;\;\omega_1+n<0\}
\]
\[
=\#\Bigl(\bZ\cap(\omega_0,\omega_1)\,\Bigr)-\#\Bigl(\bZ\cap(\omega_1,\omega_0 )\,\Bigr).
\]
We conclude
\begin{equation}
SF(A_1,A_0)=\bigl(\, \lfloor\omega_1\rfloor-\lfloor\omega_0\rfloor\,\bigr),\;\;\omega_i=\frac{A_i(L)}{2\pi}.
\label{eq:  sf}
\end{equation}
Using (\ref{eq: xi}) we deduce
\begin{equation}
SF(A_1,A_0)= \lfloor\omega_{A_1}\rfloor- \lfloor\omega_{A_0}\rfloor=\omega_{A_1}-\omega_{A_0}+\bigl(\,\xi_{A_1}-\xi_{A_0}\,\bigr).
\label{eq: sf-xi}
\end{equation}

\begin{remark}[Rescaling trick] Note that the rescaling
\[
[0,L_1]\ni \tau \mapsto t=\frac{\tau}{c}\in[0,L_0],\;\;c=\frac{L_1}{L_0}.
\]
induces an isometry $\eI_{L_1,\eL_0}: \bsH_{L_0}=L^2(0,L_0;\bC)\ra \bsH_{L_1}=L^2(0,L_1;\bC)$,
\[
\bsH_{L_0}\ni f(t) \mapsto \eI_{L_1,L_0}f(\tau):=c^{1/2}f\left(\frac{\tau}{c}\,\right)\in \bsH_{L_1}.
\]
The unbounded operator $\frac{d}{dt}$ on $\bsH_{L_0}$ is the conjugate to the operator $c\frac{d}{d\tau}$ on $\bsH_{L_1}$.  

If $\alpha(t)$ is a real bounded measurable function on $[0,L_0]$, then the  bounded operator on $\bsH_{L_0}$ defined by pointwise multiplication by $\alpha(t)$ is conjugate to the bounded operator on $\bsH_{L_1}$  defined by the multiplication by $a(\tau)=\alpha(\tau/c)$.  Hence   the  unbounded operator $D_b$ on $\bsH_{L_0}$ is conjugate to the unbounded operator $c D_{c^{-1}a}$  on $\bsH_{L_1}$,
\begin{equation}
c D_{c^{-1}a}= \eI_{L_1,L_0}D_\alpha\eI_{L_1,L_0}^{-1}.
\label{eq: conj-iso}
\end{equation}
 Its resolvent is obtained by solving the  periodic boundary value problem
\[
\ii u +c\left(-\ii \frac{d}{d\tau} +c^{-1}a(\tau) \right) u(\tau)= f(\tau),\;\;u(0)=u(L_1),
\]
or equivalently 
\[
\frac{d}{d\tau}u +c^{-1} \bigl(\, a(\tau) -\ii\,\bigr)u =c^{-1}\ii f,\;\;u(0)=u(L_1).
\]
If we set 
\[
A(\tau)=\int_0^\tau a(\si)d\si\;\;\mbox{and}\;\;\Phi_{A,c}(t)=c^{-1}\Phi_A(\tau)=c^{-1}(\ii A(\tau)-\tau),
\]
then we see that $R_\alpha$ is conjugate to the integral operator $R_{A,c}$
\[
R_{A,c} f(\tau)= \frac{c^{-1}\ii e^{-\Phi_{A,c}(\tau)}}{ e^{\Phi_{A,c}(L_1)}-1}\int_0^{L_1} e^{\Phi_{A,c}(\si)} f(s) ds+ c^{-1}\ii \int_0^t  e^{-(\Phi_{A,c}(\tau)-\Phi_A(\si)} f(\si) d\si.
\]
Arguing exactly as in the proof of Proposition \ref{prop: cont}  we deduce that if $A_n$ coverges very weakly to $A\in \eA_{L_1}$ and  the sequence of positive numbers $c_n$ converges to the positive number $c$ then $R_{A_n,c_n}$ converges in the operator norm to $R_{A,c}$.

For any $c>0$ and $A\in \eA$ we define  the operator
\[
T_{A,c}= R_{A,c}^{-1}-\ii,\;\;c>0.
\]
Note that $T_{A,c}= cT_{c^{-1}A}$. Then for every $c>0$ the spectrum  of  $T_{A,c}$ is
\[
\spec\bigl(\,T_{A,c}\,\bigr)=c\spec\bigl(\, T_{c^{-1}A}\,\bigr).\proofend
\]
\label{rem: resc}
\end{remark}

We want to  give a more intuitive description of the operators $R_A$,  and $T_A$ for a  large class of $A$'s.  We begin by introducing a nice subclass $\eA_*$ of $\eA$.  Let $H(t)$ denote the Heaviside function
\[
H(t)=\begin{cases}
1, & t\geq 0\\
0,& t<0.
\end{cases}
\]

\begin{definition} We say that $A\in \eA$ is   \emph{nice} if there exists $a\in L^\infty(0,L)$,  a finite subset $\eP\subset (0,L)$, and a  function $c:\eP\ra \bR$  such that if we define
\[
A_*(t):=\int_0^t a(s) ds
\]
then
\[
A(t)=   A_*(t)+\sum_{p\in \Delta} c(p)H(t-p),\;\;\forall  t\in [0,L].
\]
We denote by $\eA_*$ the  subcollection of  nice functions.\qed
\end{definition}

 Let us first point out   that $\eA_*$ is a vector subspace   of $\eA$. Next, observe that  $A\in \eA^*$ if and only if there exists a finite subset $\eP_A\subset (0,L)$ such that the restriction of $A$ to $[0,L]\setminus \eP$ is  Lipschitz continuous. In this case $A$  admits left and right limits at any point $t\in [0,L]$ and we define
 \[
 c:\eP_A \ra \bR,\;\;c(p)=\lim_{t\searrow p}A(t)-\lim_{t\nearrow p} A(t).
 \]
 Then
 \[
 A_*(t)= A(t)-\sum_{p\in\eP} c(p)H(t-p)
 \]
 is Lipschitz continuous, it is differentiable a.e. on $[0,L]$  and we define  $a$ to be the derivative of $A_*$.  
 
 Let us next observe that  if $A\in \eA_*$ then  the operator $T_A$ can be informally described as 
 \[
 T_A=-\ii\frac{d}{dt} +a(t)+\sum_{p\in\eP_A} c(p) \delta_p.
 \]
 In other words, $T_A$ would like to be   a  Dirac type operator  whose  coefficients are measures.    In the above informal discussion we left out a description of the domain  of $T_A$.  Below  we would like to give a precise   description of  $T_A$ as a closed unbounded selfadjoint operator defined by an elliptic boundary value problem.

For any partition   of $[0,L]$, $ \eP=\{ 0<t_1< \cdots <t_{n-1}<L\}$, we set
\[
 t_0:=0, \;\;t_n:=L,\;\;I_k :=[t_{k-1}, t_k],\;\;k=1,\dotsc, n.
 \]
We  define the Hilbert space
 \[
 \bsH_\eP=\bigoplus_{k=1}^n L^2(I_k,\bC),
 \]
 and the  Hilbert space isomorphism
 \[
 \eI_\eP:\bsH\ra \bsH_\eP,\;\; \bsH\ni f \mapsto \bigl(\,f|_{I_1},\dotsc, f|_{I_n}\,\bigr)\in \bsH_\eP.
 \]
  Let   $A\in \eA_*$  and $\eP$ be a  partition 
   \[
 \eP=\{ 0<t_1< \cdots <t_{n-1}<L\}
 \]
 that contains the set of discontinuities of $\eA$, $\eP\supset \eP_A$.  We set
 \[
 a=\frac{dA_*}{dt},\;\; ;a_k=a|_{I_k},\;k=1,\dotsc, n.
\]
For $j=1,\dotsc, n-1$ we denote by $c_j=c_j(A)$ the jump of $A$ at $t_j$,
 \[
 c_j=A(t_j^+)-A(t_j^-).
 \]
 Finally we define the closed  unbounded linear operator 
 \[
 L_{\eA,\eP}:\Dom(L_{A,\eP})\subset \bsH_\eP \ra \bsH_\eP, 
 \]
 where $\Dom(L_{A,\eP})$ consists of   $n$-uples  $(u_k)_{1\leq k\leq n}\in\bsH_\eP$ such that
 \begin{subequations}
 \begin{equation}
 u_k\in L^{1,2}(I_k),\;\;k=1,\dotsc, n,
 \label{eq: doma}
 \end{equation}
 \begin{equation}
u_{j+1}(t_j)=e^{-\ii c_j}u_j(t_j),\;\;j=1,\dotsc, n-1,
\label{eq: domb}
\end{equation}
\begin{equation}
 u_n(L)=u_1(0).
 \label{eq: domc}
 \end{equation}
 \end{subequations}
 and
 \begin{equation}
 L_{A,\eP}(u_1,\dotsc, u_n)= \Bigl(\, -\ii\frac{du_1}{dt}+a_1u_1,\dotsc,  -\ii\frac{du_n}{dt}+a_nu_n\,\Bigr).
 \label{eq: op}
 \end{equation}
 A standard argument shows that  $L_{A,\eP}$ is closed, densely defined and selfadjoint.  In particular, the operator  $(L_{A,\eP}+\ii)$ is invertible, with bounded inverse.

 \begin{theorem}  For any $A\in \eA_*$  and any partition 
 \[
  \eP=\{ 0<t_1< \cdots <t_{n-1}<L\}
  \]
  that contains the set of discontinuities of $\eA$ we have the equality
 \[
 L_{A,\eP}=\eI_\eP T_A\eI_\eP^{-1}.
 \]
 \label{th: split}
 \end{theorem}

 \begin{proof} For simplicity we write $L_A$ instead of $L_{A,\eP}$. We will prove the equivalent statement
 \[
 (\ii+L_A)^{-1} = \eI_\eP (T_A+\ii)^{-1}\eI_\eP^{-1}=\eI_\eP R_A\eI_{\eP}^{-1}.
 \]
 In other words we have to prove that  for any  $u, f\in \bsH$  if  $u=R_Af$, then $u\in \Dom(L_A)$ and  $(L_A+\ii)\eI_\eP u= \eI_\eP f$. More precisely, we have to show that  the collection $\eI_Au=(u_k)_{1\leq k\leq n}$ satisfies (\ref{eq: doma}--\ref{eq: domc})  and (\ref{eq: op}). Using  (\ref{eq: r1}) we deduce
 \begin{equation}
u(t) = \frac{\ii e^{-\Phi_A(t)}}{ e^{\Phi_A(L)}-1}\int_0^L e^{\Phi_A(s)} f(s) ds+ \ii e^{-\Phi_A(t)}\int_0^t  e^{\Phi_A(s)} f(s) ds.
 \label{eq: r3}
 \end{equation}
 This implies   the condition (\ref{eq: doma}).  The  condition  (\ref{eq: op}) follows by direct computation using (\ref{eq: r3}).

 Next, we observe  that
 \[
 u(t_j^+)=\frac{\ii e^{-\Phi_A(t_j^+)}}{ e^{\Phi_A(L)}-1}\int_0^L e^{\Phi_A(s)} f(s) ds+\ii e^{-\Phi_A(t_j^+)}\int_0^{t_j}  e^{\Phi_A(s)} f(s) ds,
 \]
 \[
 u(t_j^-)=\frac{\ii e^{-\Phi_A(t_j^-)}}{ e^{\Phi_A(L)}-1}\int_0^L e^{\Phi_A(s)} f(s) ds+\ii e^{-\Phi_A(t_j^-)}\int_0^{t_j}  e^{\Phi_A(s)} f(s) ds,
 \]
 from which we conclude that
 \[
 u(t_j^+)= e^{-\ii\bigl(\,\Phi_A(t_j^+)-\Phi_A(t_j^-)\,\bigr)}u(t_j^-),\;\;\forall j=1,\dotsc n-1.
  \]
  This proves (\ref{eq: domb}).  The equality  (\ref{eq: domc})   follows directly from (\ref{eq: r2}). \end{proof}

 \begin{remark} We would like to place the above  operator $L_A$  in a broader perspective that  we will use  extensively in Section \ref{s: 4}.  Consider   a compact, oriented   $1$-dimensional manifold with boundary $I$. In other words $I$ is a disjoint union of finitely many compact intervals
 \[
 \bsI=\sqcup_{k=1}^n I_k.
 \]
 If $I_k:=[a_k,b_k]$, $a_k<b_k$, then we set 
 \[
 \pa_+I_k :=\{b_k\}, \;\;\pa_- I_k:=\{a_k\},\;\;\pa_+\bsI:=\{b_1,\dotsc, b_n\},\;\;\pa_- \bsI:=\{a_1,\dotsc, a_n\}.
 \]
 In particular, we have a direct sum decomposition of (finite dimensional) Hilbert spaces
 \[
 \bsE:=L^2(\pa I,\bC)=L^2(\pa_+ \bsI)\oplus L^2(\pa_- \bsI)=:\bsE_+\oplus \bsE_-.
 \]
 On the space $C^\infty(\bsI,\bC)$  of smooth  complex valued  functions on $\bsI$ we have a canonical,  symmetric Dirac $\eD$ operator   described  on each $I_k$ by $-\ii\frac{d}{dt}$. Let $\sigma$ denote the principal symbol of this operator. If $\bnu_*$ denotes the  \emph{outer} conormal  to the boundary. We then get  an operator
 \[
 J=\si(\bnu_*):L^2(\pa \bsI,\bC)\ra L^2(\pa \bsI,\bC).
 \]
 It is a unitary operator satisfying $J^2=-1$, $\ker(\ii+J)=\bsE_+$, and  $\ker(\ii-J)=\bsE_-$. It thus defines a Hermitian symplectic structure in the sense of \cite{Ar1, Cib, N1}. A  (hermitian) lagrangian subspace of $\bsE$ is then  a complex subspace $L$ such that $L^\perp=JL$. We denote by $\Lag(\bsE,J)$ the Grassmannin of hermitian lagrangian spaces.   We denote by $\Iso(\bsE_+,\bsE_-)$ the  space of  linear isometries $\bsE_+\ra \bsE_-$. As explained in \cite{Ar1} there exists a natural bijection\footnote{There are various conventions in the definition of this bijection. We follow the conventions in \cite{Cib}.}
 \[
 \Iso(\bsE_+,\bsE_-)\ra \Lag(\bsE),\;\; \Iso(\bsE_+,\bsE_-)\ni T\longmapsto \Gamma_T
 \]
 where $\Gamma_T$  is the graph of $T$ viewed as a subspace of  $\bsE$. Our spaces $\bsE_\pm$ are equipped with natural bases and through these  bases we can identify $\Iso(\bsE_+,\bsE_-)$ with the unitary group $U(n)$.  We denote by $\Delta$ the Lagrangian subspace  corresponding to the identity operator.
 
 Any subspace $V\subset \bsE$ defines a  Fredholm operator
 \[
 \eD_V: \Dom(\eD_V)\subset L^2(\bsI,\bC)\ra L^2(\bsI,\bC),
 \]
 where
 \[
 \Dom(\eD_V) =\bigl\{\, u\in L^{1,2}(\bsI,\bC);\;\;u|_{\pa I}\in V\,\bigr\},\;\;\eD_V u=\eD u.
 \]
 The index of this operator  is
 \[
 i_V = \dim(V\cap \Delta) -\dim (\Delta\cap JV^\perp)= \dim(V\cap \Delta) -\dim (J\Delta\cap V^\perp)
 \]
 \[
 = \dim(V\cap \Delta) -\dim (\Delta^\perp\cap V^\perp)=\dim V-\codim \Delta= \dim V-n.
 \]
 A simple  argument shows that $\eD_V$ is selfadjoint if and only if $V\in \Lag(\bsE)$.  As we explained  above,  in this case $V$ can be identified with the graph of an isometry $T:\bsE_+\ra \bsE_-$. We say that $T$   is the \emph{transmission  operator} associated to the  selfadjoint boundary value problem.
 
 For example, if in Theorem \ref{th: split} we let $ A(t)=\sum_{j=1}^{n-1} c_jH(t-t_j)$, then we see   that the operator $L_A$ can  be identified with the operator $\eD_{\Gamma_T}$, where the transmission operator  $T\in \Iso(\bsE_+,\bsE_-)$ is given by the unitary $n\times n$ matrix
 \[
 T=\left[\begin{array}{llrccc}
 0 & 0 & 0 & \cdots & 0 & 1\\
 e^{-\ii c_1} &0 & 0&\cdots & 0 & 0\\
 0 & e^{-\ii c_2} & 0 &\cdots &0 & 0\\
 \vdots & \vdots & \vdots & \vdots & \vdots &\vdots\\
  \cdot & \cdot & \cdot & \cdot & \cdot &\cdot\\
 0 &0 & 0& \cdots & e^{-\ii c_{n-1}} & 0
 \end{array}
 \right].
 \]
\qed
\label{rem: bvp}
 \end{remark}

\section{The Atiyah-Patodi-Singer theorem}
\label{s: 2}
\setcounter{equation}{0}
We     review    here the Atiyah-Patodi-Singer index theorem for  Dirac operators on  manifold with boundary,    when the metric is not assumed to be   cylindrical   near the boundary.    Our presentation  follows closely,  \cite{Gil, Grubb}, but we    present  a few more details since  the various orientation  conventions  and the terminology in \cite{Gil, Grubb} are different from those in  \cite{BGV, N} that we use throughout this paper.

Suppose  $(\widehat{M}, \hat{g})$ is a compact, oriented Riemann, and  $M\subset \widehat{M}$ be  a  hypersurface in $\widehat{M}$ co-oriented by a unit normal vector field $\bnu$  along $M$. Let $n:=\dim M$ so that $\dim\widehat{M}=n+1$. We denote by $g$  the induced metric on $M$.   We first want to define  a canonical restriction to $M$  of a Dirac operator on $\hat{M}$.

 Let $\exp^{\hat{g}}: T\widehat{M}\ra \widehat{M}$ denote the exponential map determined by the metric $\hat{H}$.   For sufficiently small $\ve>0$ the map
\[
(-\ve,\ve)\times M\ni (t,p)\mapsto  \exp^{\hat{g}}_p\bigl(\,t\bnu(p)\,\bigr)\in \widehat{M}
\]
 is a diffeomorphism onto  a small open tubular neighborhood $\eO_\ve$ of $M$.     The metric  $g$ determines a cylindrical metric $dt^2+g$ on $(-\ve,\ve)\times M$.   Via  the above diffeomorphism  we get a metric $\hat{g}_0$ on $\eO_\ve$. We say that  $\hat{g}_0$ is the \emph{cylindrical approximation} of $\hat{g}$ near $M$. 

  We denote by $\widehat{\nabla}$ the Levi-Civita connection  of the metric  $\hat{g}$ and by $\widehat{\nabla}^0$ the Levi-Civita connection of the metric $\hat{g}_0$.  We set
 \[
 \bXi:=\widehat{\nabla}-\widehat{\nabla}^0\in \Omega^1\bigl(\,\eO_\ve,\,\End (\, T\widehat{M}\,)\,\bigr).
 \]
 To get a more explicit description of $\bXi$ we fix a local oriented, $g$-orthonormal frame $(\be_1,\dotsc, \be_n)$ on $M$.  Together with the unit  normal vector field $\bnu$  we obtain a local   oriented orthonormal frame  $(\bnu,\be_1,\dotsc, \be_n)$ of $T\widehat{M}|_{M}$. We extend  it by parallel transport  along the geodesics orthogonal to $M$ to  a local, oriented orthonormal    frame   $( \hat{\bnu},\hat{\be}_1,\dotsc, \hat{\be}_n)$ of $T\widehat{M}$.

 Denote by $\widehat{\bom}$ the connection form   associated to $\widehat{\nabla}$  by this frame, and by $\widehat{\bthet}$ the connection form   associated to $\widehat{\nabla}^0$  by this frame. 
 We can represent both $\widehat{\bom}$ and $\widehat{\bthet}$ as  skew-symmetric $(n+1)\times (n+1)$ matrices
 \[
 \widehat{\bom}=\bigl(\, \widehat{\bom}^i_j\,\bigr)_{0\leq i,j\leq n},\;\;\widehat{\bthet}=\bigl(\, \widehat{\bthet}^i_j\,\bigr)_{0\leq i,j\leq n},
 \]
where the entries are $1$-forms. Then $\bXi=\widehat{\bom}-\widehat{\bthet}$.

We set $\hat{\be}_0:=\hat{\bnu}$,  and  we denote by $(\hat{\be}^k)_{0\leq k\leq n}$ the dual orthonormal frame of $T^*\widehat{M}$.Then we have
\[
\hat{\bom}^i_j=\hat{\bom}^i_{kj}\hat{\be}^k,\;\; \hat{\bthet}^i_j=\hat{\bthet}^i_{kj}\hat{\be}^k,\;\;\widehat{\nabla}_k\hat{\be}_j =\hat{\bom}^i_{kj}\hat{\be}_i,\;\; \widehat{\nabla}_k^0\hat{\be}_j =\hat{\bthet}^i_{kj}\hat{\be}_i,\;\;\forall 0\leq j,k\leq n,
\]
where we  have used  Einstein's summation convention.

Observe that $\widehat{\nabla}^0\hat{\be}_0=0$ so that $\hat{\bthet}_0^i=\hat{\bthet}^0_i=0$.  Also,
\[
\hat{\bom}^i_{jk}=\hat{\bthet}^i_{jk},\;\;\forall 1\leq i,j,k\leq n.
\]
If we write     
\[
\bXi=\bigl(\bXi^i_j\,\bigr)_{0\leq i,j\leq n},\;\;\bXi^i_j=\bXi^i_{jk}\hat{\be}^k,
\]
and we let $o(1)$ denote any quantity that vanishes along  $M$. then we have
\begin{equation}
\bXi^i_j=-\bXi^j_i,\;\;\forall 0\leq i,j\leq n,
\label{eq: vanish}
\end{equation}
\begin{equation}
\bXi^i_{kj}=o(1),\;\;\forall 1\leq i,j\leq n,\;\;0\leq k\leq n.
\label{eq: sym}
\end{equation}
We set
\[
\bXi_{kij}:= \bXi^i_{kj}=\hat{g}\bigl( \widehat{\nabla}_k\hat{\be}_j, \hat{\be}_i\,\bigr),\;\;\hat{\bom}_{ij}=\hat{\bom}^i_j,\;\;\hat{\bthet}_{ij}=\hat{\bthet}^i_j,\;\; \bom_{kij}=\bom^i_{kj},\;\;\bthet_{kij}=\bom^i_{kj}.
\]
We denote by $Q$   the second fundamental form\footnote{Our  definition  of the second fundamental form differes by a  sign from the usual definition. With our definition the  round sphere  cooriented by the outer normal has positive mean curvature.} of the embedding $M\hra \widehat{M}$,
\[
Q(e_i,e_j)= g(\nah_{\be_i}\bnu,\be_j).
\]
 Along the boundary we have the equalities
\begin{subequations}
\begin{equation}
\Xi_{kj0}=\bXi_{jk0}=-\bXi_{k0j}=Q(\be_j,\be_k)\;\;\forall 1\leq i,j\leq n,
\label{eq:  second}
\end{equation}
\begin{equation}
\Xi_{ij0}=0,\;\;\forall 0\leq i,j\leq n .
\label{eq: van2}
\end{equation}
\end{subequations}

To understand  the  nature of the   restriction to a hypersurface  of a Dirac operator we begin with a special case.  Namely, we assume that $\widehat{M}$ is equipped with a $spin$ structure.   We denote by  $\hat{\bS}$ the associated  complex spinor bundle so that $\widehat{\bS}$ is $\bZ/2$-graded is $\dim \widehat{M}$ is even, and ungraded  otherwise.     We have a Clifford multiplication
\[
\hbc: T^*M\ra \End (\hat{\bS}).
\]
The metrics $\hat{g}$     and $\hat{g}_0$ define connections $\widehat{\nabla}^{spin}$ and $\widehat{\nabla}^{spin,0}$ on $\hat{\bS}|_{\eO_\ve}$. Using the local frame $(\hat{\be}_i)_{0\leq i, j\leq n}$   we can write
\[
\widehat{\nabla}^{spin}_k =\pa_k-\frac{1}{4} \hat{\bom}_{kij}\hbc(\hat{\be}^i)\hbc(\hat{\be}^j),\;\;\widehat{\nabla}_k^{spin,0}=\pa_k-\frac{1}{4} \hat{\bthet}_{kij}\hbc(\hat{\be}^i)\hbc(\hat{\be}^j),
\]
where we again use Einstein's  summation convention.

Using the     connections $\widehat{\nabla}^{spin}$ and $\widehat{\nabla}^{spin,0}$ we obtain two Dirac operators $\hat{D}$ and respectively $\hat{D}_0$ on  $\widehat{\bS}|_{\eO_\ve}$
\[
\hat{D}= \sum_{i=0}^n\hbc(\hat{e}^i)\widehat{\nabla}^{spin}_i,\;\; \hat{D}_0=\sum_{i=0}^n \hbc(\hat{e}^i)\widehat{\nabla}^{spin,0}_i.
\]
Identifying $\eO_\ve$ with $(-\ve,\ve)\times M$  we obtain a projection 
\[
\pi:\eO_\ve \ra M.
\]    
We set  $\bS:=\hat{\bS}|_M$. The parallel transport  given by $\nah^{spin}$ yields  a bundle isomorphism $\hat{\bS}|_{\eO_\ve}\cong \pi^*\bS$. Using these identifications  we can rewrite  the operators $\hat{D}$ and $\hat{D}_0$    as
\[
\hat{D}=\hbc(\hat{\be}_0)\bigl(\,\widehat{\nabla}^{spin}_0- D(t)\,\bigr): C^\infty(\pi^*\bS)\ra C^\infty(\pi^* \bS),
\]
\[
\hat{D}_0=\hbc(\hat{\be}^0)\bigl(\,\pa_0- D_0(t)): C^\infty(\pi^*\bS)\ra C^\infty(\pi^* \bS).
\]
The operators $D(t)$ and $D_0(t)$ are first order   differential operators $C^\infty(\widehat{\bS}|_{\{t\}\times M})\ra  C^\infty(\widehat{\bS}|_{\{t\}\times M})$ and thus  can be viewed as   $t$-dependent operators  on  $\bS$. 

The operator $D_0(t)$ is in fact independent of $t$ and thus we can identify it with a Dirac operator  on $C^\infty(\bS)\ra C^\infty(M)$.  It is called  the \emph{canonical restriction} of $\hat{D}$ to $M$, and   we will denote it by $\eR_M(\hat{D})$.This operator is intrinsic to  $M$.  More precisely  when  $\dim \widehat{M}$ is even then $\bS$ is the direct sum of two copies  of the spinor bundle  on $M$ and the operator $D_0$ is the direct sum of two copies of the $spin$-Dirac operator  determined by the Riemann metric on  $M$.

When $\dim\widehat{M}$ is odd  then $\bS$ is the spinor bundle on $M$ and $D_0$ is the $spin$-Dirac operator  determined by the metric on the boundary and the induced  $spin$ structure.  We would like to express $\eR_M(\hat{D})$ in terms of $D(t)|_{t=0}$.

Let $\bnu_*:=\hbe^0\in  C^\infty\bigl(\, T^*\widehat{M}|_{\pa \widehat{M}}\,\bigr)$, set  $J:=\hbc(\bnu_*)$ and define $\bc:  T^*M\ra \End(\bS)$ by setting
\[
\bc(\alpha)=\hbc(\,\bnu_*(p)\,)\hbc(\alpha)=J\hbc(\alpha),\;\;\forall p\in M,\;\;\alpha\in T^*M\subset (T^*\widehat{M})|_M.
\]
Observe first that
\[
\eR_M(\hat{D})=D_0(t)=  \pa_0 +J \hat{D}_0.
\]
Next we observe that
\[
\nah^{spin}-\nah^{spin,0}= -\frac{1}{4}\sum_{i,j}\bXi_{ij}\hbc(\hat{\be}^i)\hbc(\hat{\be}^j).
\]
so that
\[
\nah^{spin}_0-\nah^{spin,0}_0=\nah^{spin}_0-\pa_0= -\frac{1}{4}\Xi_{0ij}J\hbc(\hbe^i)\hbc(\hbe^j)=o(1),
\]
\[
\hat{D}-\hat{D}_0= -\frac{1}{4}\sum_{i,j,k} \bXi_{kij}\hbc(\hat{\be}^k)\hbc(\hat{\be}^i)\hbc(\hat{\be}^j)=:\eL.
\]
We denote by  $\eL(t)$ the restriction of $\eL$ to the slice $\{t\}\times M$ so that $\eL(t)$ is an endomorphism of $\widehat{\bS}|_{\{t\}\times M}$.

Hence
\[
\hat{D}=J\pa_0- JD(t) ,\;\;D(t)= J\hat{D}+\pa_0=  J\hat{D}_0+\pa_0+ J\eL =D_0(t)+J\eL,
\]
so that
\[
D(0)=\eR_M(\hat{D})+J\eL(t)|_{t=0}.
\]
Thus, we need to compute  the endomorphism $J\eL(t)|_{t=0}$. We have
\[
J\eL=-\frac{1}{4}\sum_{i,j,k}J\bXi_{kij}\hbc(\hbe^k)\hbc(\hbe^i)\hbc(\hbe^j).
\]
There are many cancellations in the above sum.  Using (\ref{eq: sym}) we deduce that the   terms corresponding to $k=0$ vanish. Using (\ref{eq: vanish}) we deduce that the terms corresponding to $i,j>0$ or $i=j$ also vanish along the boundary. Thus
\[
J\eL=-\frac{1}{4}\sum_{i\neq j,k\neq 0}\bXi_{kij}J\hbc(\hbe^k)\hbc(e^i)\hbc(e^j)+o(1)=-\frac{1}{2}\sum_{i>j, k>0}\bXi_{kij}J\hbc(\hbe^k)\hbc(e^i)\hbc(e^j)+o(1)
\]
\[
=-\frac{1}{2}\sum_{i>0,k>0}\bXi_{ki0}J\hbc(\hbe^k)\hbc(\hbe^i)\hbc(\hbe^0)+o(1).
\]
Using the equalities $J=\hbc(\hbe^0)$, $J\hbc(\be^\ell)=-\hbc(\hbe^k)J$ for $\ell>0$ we deduce
\[
J\eL=\frac{1}{2}\sum_{i,k>0}\bXi_{ki0}\hbc(\be^k)\hbc(\hbe^j)=-\frac{1}{2}\sum_{j>0}\bXi_{ii0}+o(1)=-\frac{1}{2}\tr Q.
\]
The scalar $\tr Q$ is the  mean curvature of  $M\hra \widehat{M}$ and we denote it by $h_M$.  Hence
\begin{equation}
D(t)|_{t=0}= \eR_M(\widehat{D})-\frac{1}{2}h_M.
\label{eq: red}
\end{equation}
A similar equality was  proved   in \cite[Lemma 4.5.1]{KM}, although  in \cite{KM}  they use  a different definition for the induced Clifford  multiplication on the boundary that leads to  some sign differences.

If now  $\widehat{E}\ra \widehat{M}$ is a hermitian vector bundle over  $\widehat{M}$ and $\nah^E$ is  a Hermitian connection on $\widehat{E}$ then we obtain in standard  fashion a twisted Dirac operator $\hat{D}_E: C^\infty(\widehat{\bS}\otimes \widehat{E})\ra C^\infty(\widehat{\bS}\otimes\widehat{E})$. Using the parallel transport given by $\nah^E$ we obtain an isomorphism
\[
\widehat{E}|_{\eO_\ve}\cong \pi^*E,\;\;\mbox{where}\;\;E:=\widehat{E}|_{M}.
\]
Along $\eO_\ve$ the operator $\widehat{D}_E$ has the form
\[
\widehat{D}_E=J(\pa_t -D_E(t)).
\]
If on $\eO_\ve$ we replace the metric $\hat{g}$ with its cylindrical approximation $\hat{g}_0$ we obtain a new    Dirac operator
\[
\widehat{D}_{E,0}: C^\infty\bigl(\, \pi^*(\bS\otimes E)\,\bigr)\ra C^\infty(\,\pi^*\bigl(\bS\otimes E)\,\bigr)
\]
which along the boundary has the form $J(\pa_t -D_{E_0})$, where $D_{E,0}: C^\infty(\bS\otimes E)\ra C^\infty(\bS\otimes E)$. We set $\eR_M(\widehat{D}_E):= D_{E,0}$ and as before we obtain the identity
\begin{equation}
D_E(t)|_{t=0}=\eR_M(\widehat{D}_E)-\frac{1}{2}h_M.
\label{eq: red1}
\end{equation}
This is a purely local result so that a similar formula  holds for the geometric Dirac operators determined by a $spin^c$ structure.

We want to apply the above discussion to a  very special case. Consider a compact oriented surface $\Sigma$ with possibly disconnected boundary $\pa \Sigma$. \emph{We think of $\pa \Sigma$ as a hypersurface in $\Sigma$ cooriented  by the outer normal.} 

Fix   a Riemann metric $\hg$ on $\Sigma$, smooth up to the boundary.    Denote  by $s$ the arclength  coordinate  on a component $\pa_0\Sigma$ of the  boundary.       As before we can identify an open neighborhood $\eO$ of this component of the  boundary with a cylinder  $(-\ve,0]\times  S^1$.  In this   neighborhood the metric $\hg$ has the form
\[
\hg= dt^2+ w^2 ds^2
\]
where $w:(-\ve,0]\times S^1\ra (0,\infty)$ is a smooth positive function in the variables  $t,s$ such that $w(0,s)=1$, $\forall s$.

 The metric and the orientation on $\Sigma$ defines  an integrable almost complex structure  $J:T\Sigma\ra T\Sigma$. More precisely, $J$ is given by the  counterclockwise rotation by $\pi/2$.  We denote by $K_\Sigma$ the canonical complex line bundle determined by $J$.       We get  a  Dolbeault  operator
 \[
 (\bpar+\bpar^*) : C^\infty(\uc_\Sigma\oplus K_\Sigma^{-1})\ra C^\infty(\uc_\Sigma\oplus K_\Sigma^{-1}).
 \]
 We regard this as the  Dirac operator defined by  the metric $\hg$, a $spin^c$ structure. The twisting line bundle is $K_{\Sigma}^{-1/2}$, where the connection on $K_\Sigma$ is the connection induced by the Levi-Civita  connection of the metric $\hg$.  We analyze the form of $\bpar: C^\infty(\uc_\Sigma)\ra C^\infty(K_\Sigma^{-1}) $on the cylindrical neighborhood  $\eO$.   We set
 \[
 \be^0=dt,\;\; \be^1=wds.
 \]
 Then $\{\be^0,\be^1\}$ is an oriented, orthonormal frame of $T^*\Sigma|_{\eO}$.  We denote by $\{\be_0,\be_1\}$ its dual frame of $T\Sigma$.     We let $\bc:T^*\Sigma \ra \End(\uc_\Sigma\oplus K_\Sigma)$ be the Clifford multiplication    normalized by the  condition that the operator $dV:=\bc(\be^0)\bc(\be^1)$ on $\uc_\Sigma\oplus K_\Sigma^{-1}$ has the block decomposition \cite[\S 3.2]{BGV},
 \begin{equation}
 \bc(\be^0)\bc(\be^1)=\left[
 \begin{array}{cc}
 -\ii &0\\
 0&\ii
 \end{array}
 \right].
 \label{eq: grad}
 \end{equation}
The Levi-Civita  induces a natural connection on on $K_\Sigma^{-1}$ and if we use the trivial connection on  $\uc_\Sigma$ we get a connection $\nabla$ on $\uc_\Sigma\oplus K_\Sigma^{-1}$. The associated  Dirac operator   is  $D_\Sigma=\sqrt{2}(\bpar+\bpar^*)$.     The even part   of this operator is
\[
D^+=\sqrt{2}\bpar: C^\infty(\uc_\Sigma)\ra C^\infty(K_\Sigma^{-1}).
\]
We want to compute its canonical restriction to the boundary. 

The Levi-Civita connection    $\nah$  determined by $\hg$ is described on $\eO$ by a $1$-form $\omega$  uniquely determined by  Cartan's structural equations
 \[
\nah=d+ \left[
\begin{array}{cc}
0 & -\omega\\
\omega & 0
\end{array}
\right],\;\; d\left[
 \begin{array}{c}
 e^0\\
 \be^1
 \end{array}
 \right]=\left[
 \begin{array}{cc}
 0 & \omega\\
 -\omega &0
 \end{array}
 \right]\wedge d\left[
 \begin{array}{c}
 e^0\\
 \be^1
 \end{array}
 \right].
 \]
 We deduce $\omega= a \be^1$, $a\in C^\infty(\eO)$ and from the  equality
 \[
 \frac{w'_t}{w} \be^0\wedge \be^1= w'_t dt\wedge ds = d\be^1= a\be^0\wedge \be_1
 \]
 we conclude  $a =\pa_t\log w$ so that
 \[
 \omega=  \pa_t(\log w) \be^1 = w'_t ds.
 \]
The  mean curvature $h$  of the boundary component $\pa_0\Sigma$ is the restriction to $t=0$ of the function $w'_t$.  The Riemann curvature is  described by the  matrix
 \[
 \left[
 \begin{array}{cc}
 0 & -d\omega\\
 d\omega &0
 \end{array}
 \right]= \left[
 \begin{array}{cc}
 0 & - w''_{t}dt \wedge ds\\
 w''_tdt\wedge ds &0
 \end{array}
 \right]= \left[
 \begin{array}{cc}
 0 & -\frac{w''_t}{w}\\
 \frac{w''_t}{w}& 0
 \end{array}
 \right] \be^0\wedge \be^1.
 \]
 If we denote by $\pa$ the trivial connection on $\uc_\Sigma$ then we deduce
 \[
 D^+_\Sigma:=\bc(\be^0)\pa_{\be_0}+\bc(\be^1)\pa_{\be_1}= \bc(\be_0)\bigl( \pa_t -\bc(\be^0)\bc(\be^1)\pa_{\be_1}\,\bigr)
 \]
 so that
 \[
 D^+_\Sigma(t)=\bc(\be^0)D^+_\Sigma+\pa_t= \bc(\be^0)\bc(\be^1)\pa_{\be_1}\stackrel{(\ref{eq: grad})}{=}-\ii\pa_{\be_1}.
 \]
 Above, the operator $D^+_\Sigma(t)$ is,  \emph{canonically}, a differential operator
 \[
 D^+_\Sigma(t): C^\infty(\uc_{\pa\Sigma})\ra C^\infty(\uc_{\pa \Sigma}),
 \]
 where $\uc_{\pa\Sigma}$ denotes the trivial complex line bundle over $\pa\Sigma$. The boundary restriction  is then according to (\ref{eq: red1})
 \begin{equation}
\eR_{\pa\Sigma}(\bpar)=D^+_\Sigma(t)+\frac{1}{2}h= -\ii\pa_{\be_1}+\frac{1}{2} h.
 \label{eq: br}
 \end{equation}
 Let us observe that along the boundary we have $\pa_{\be_1}=\pa_s$.

Consider the Atiyah-Patodi-Singer  operator
\[
\bpar_{APS}:\Dom(\bpar_{APS})\subset L^2(\Sigma)\ra L^2(\Sigma),\;\;\bpar_{APS}u=\bpar u,\forall u\in \Dom(\bpar_{APS}),
\]
where
\[
\Dom(\bpar_{APS})=\{u\in L^{1,2}(\Sigma, \bC);\;\;u|_{\pa \Sigma}\in \Lambda^-_\pa\,\bigr\},
\]
and  $\Lambda^-_\pa$ is the closed subspace of $L^2(\pa \Sigma)$ generated by the eigenvectors of the  operator $B:=\eR_{\pa\Sigma}(\bpar)$ corresponding to  strictly negative eigenvalues.

The index theorem of \cite{Gil, Grubb} implies  $\bpar_{APS}$ is Fredholm and
\[
i_{APS}(\Sigma,g):={\rm index}\,(\bpar_{APS})=\frac{1}{2}\int_\Sigma c_1(\Sigma,g) -\xi_B,\;\;\xi_B:=\frac{1}{2}\bigl(\,\dim B+\eta_{B}(0)\,\bigr).
\]
Above,  $c_1(\Sigma, g)\in \Omega^2(\Sigma)$ is the  $2$-form $\frac{1}{2\pi} K_gdV_g$, where $K_
g$ denotes the sectional curvature of $g$ and $dV_g$ denotes the metric volume form on $\Sigma$. From the Gauss-Bonnet theorem for manifolds with boundary \cite[\S 6.6]{O} we deduce
\[
\int_\Sigma c_1(\Sigma, g) + \frac{1}{2\pi} \int_{\pa \Sigma} h ds=\chi(\Sigma),
\]
where $h:\pa \Sigma \ra \bR$ is the mean curvature function defined as above.  We deduce
\begin{equation}
i_{APS}(\Sigma,g) =\frac{1}{2}\chi(\Sigma)-\frac{1}{4\pi}\int_\Sigma h ds-\xi_B.
\label{eq: ind}
\end{equation}
If $\pa \Sigma$ has several components $\pa \Sigma=\pa_1\Sigma\sqcup\cdots \sqcup \pa_n\Sigma$, then we have  $n$ scalars
\[
H_i=\frac{1}{4\pi}\int_{\pa_i\Sigma} h ds,
\]
and a direct sum decomposition $B=\oplus_{i=1}^n B_i$, where each of the operators $B_i$ is described by (\ref{eq: br}). We set
\[
\rho_i=H_i-\lfloor H_i\rfloor,\;\;i=1,\dotsc, n.
 \]
 Then   using (\ref{eq: br})  and (\ref{eq: xi}) we  deduce
 \[
 \xi_{B_i}=\frac{1}{2}\bigl(\, \dim\ker B_i+\eta_{B_i}(0)\,\bigr)=\frac{1}{2}\bigl(\,1-2\rho_i\,\bigr).
  \]
 We can rewrite (\ref{eq: ind}) as
 \begin{equation}
 i_{APS}(\Sigma,g) =\frac{1}{2}\chi(\Sigma)-\sum_{i=1}^n H_i-\frac{1}{2}\sum_{i=1}^n\bigl(\,1-2\rho_i\,\bigr)=\frac{1}{2}\bigl(\, \chi(\Sigma)-n\,\bigr)-\sum_{i=1}^n\lfloor H_i\rfloor.
 \label{eq: ind1}
 \end{equation}

\section{Dolbeault operators on  two-dimensional cobordisms}
\label{s: 3}
\setcounter{equation}{0}
  When thinking of cobordisms   we  adopt the Morse theoretic point of view.   For us  an elementary (nontrivial) $2$-dimensional cobordism    will be a pair  $(\Sigma, f)$ where $\Sigma$ is a compact,
  \emph{connected}, oriented   surface with boundary, $ f:\Sigma\ra \bR$ is a Morse function  with a unique critical point $p_0$ located in the interior of $\Sigma$ such that
  \[
  f(\Sigma)=[-1,1],\;\; f(\pa \Sigma)=\{-1,1\},\;\; f(p_0)=0.
  \]
  In more intuitive terms, an elementary cobordism  looks like one of the two pair of pants in Figure \ref{fig: cobdeg1}, where the Morse function is understood to be the altitude.
  
\begin{figure}[ht]
\centering{\includegraphics[height=1.1in,width=3.3in]{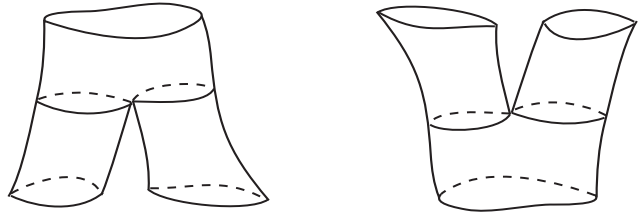}}
\caption{\sl Elementary   $2$-dimensional cobordisms.}
\label{fig: cobdeg1}
\end{figure}
We set
\[
\pa_\pm\Sigma :=f^{-1}(\pm 1).
\]
  In the sequel, for simplicity, we will assume that $\pa_+\Sigma$ is connected, i.e.,  the pair $(\Sigma, f)$ looks like the left-hand-side  of Figure \ref{fig: cobdeg1}.
  
  We fix a Riemann metric $g$  on $\Sigma$. For simplicity\footnote{The results to follow do not require the simplifying assumption (\ref{eq: can}) but  the computations would be less transparent.}  we assume   that in an open neighborhood $\eO$  near $p_0$ there exist local coordinates such that, in these coordinates we have
  \begin{equation}
  g=dx^2+dy^2,\;\;f(x,y)= - \alpha x^2 +\beta y^2,
 \label{eq: can}
 \end{equation}
  where $\alpha,\beta$ are positive   constants.  We  let $\nabla f$ denote the gradient of $f$ with  respect to this metric and we set
  \[
 C_t :=f^{-1}(t),\;\; t\neq 0.
  \]
  For $t\neq 0$  we regard $C_t$ cooriented by the gradient $\nabla f$.  Observe that $C_t$ has two connected components when $t<0$. We let $h_t:C_t\ra \bR$ be the mean curvature of this cooriented surface.  For $t\neq 0$  we set
  \[
  L_t=\int_{C_t} ds= {\rm length}\,(C_t),\;\;\omega_t:=\frac{1}{4\pi}\int_{C_t} h_t ds.
  \]
Observe that even the singular level set $C_0$ is equipped with a natural  measure defined by the arclength measure on $C_0\setminus \{0\}$. The  length of $C_0$ is finite since in a neighborhood of the singular point $p_0$ the level set isometric to a pair of intersecting  line segments in an  Euclidean space.

Denote by $W^\pm$ the stable/unstable manifolds of $p_0$ with respect to the flow $\Phi^t$ generated by $-\nabla f$. The  unstable manifold intersects the region $\{-1\leq f<0\}$ in two smooth paths (see Figure  \ref{fig: cobdeg2})
\[
[-1,0)\ni t\mapsto a_t,\;b_t\in C_t,\;\;\forall t\in [-1,0),
\]
while the stable manifold intersects the  region $\{0<f\leq 1\}$ in two smooth paths (the top red arcs in Figure \ref{fig: cobdeg2})
\[
(0,1]\ni t \mapsto a_t,\;b_t\in C_t,\;\;\forall t\in (0,1].
\]
Observe that $\lim_{t\ra 0}a_t=\lim_{t\ra 0} b_t=p_0$.  For this reason we set $a_0=b_0=p_0$.

\begin{figure}[h]
\centering{\includegraphics[height=2.3in,width=2.5in]{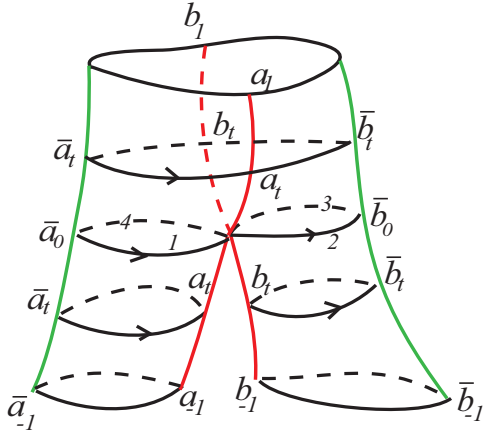}}
\caption{\sl Cutting an elementary   $2$-dimensional cobordism.}
\label{fig: cobdeg2}
\end{figure}

As we have mentioned before, for $t<0$ the  level set $C_t$ consists of two  curves. We denote by $C_t^a$ the component containing the point $a_t$ and by $C_t^b$ the component containing $b_t$.   For $t<0$ we set
\[
L_t^a:=\int_{C_t^a}ds,\;\; L_t^b:= \int_{C_t^b}ds,\;\;\omega_t^a:=\frac{1}{4\pi}\int_{C_t^a}h_t ds,\;\; \omega_t^b:=\frac{1}{4\pi}\int_{C_t^b}h_t ds
\]
so that
\[
L_t=L_t^a+L_t^b,\;\;\omega_t=\omega_t^a+\omega_t^b,\;\;\forall t<0.
\]
Fix a point $\bar{a}_{-1}\in C_{-1}^a\setminus\{a_{-1}\}$  and a point $\bar{b}_{-1}\in C^b_{-1}\setminus\{b_{-1}\}$.   For $t\in [-1,1]$  we denote by $\bar{a}_t$ (respectively $\bar{b}_t$)  the intersection of $C_t$ with the negative gradient  flow line through $\bar{a}_{-1}$ (respectively $\bar{b}_t$).  We obtain in this fashion two smooth maps (see  Figure \ref{fig: cobdeg2})
\[
\bar{a},\bar{b}:[-1,1]\ra \Sigma.
\]
For $t>0$ we denote by $I_t^a$ the component of $C_t\setminus \{a_t,b_t\}$ that contains the point $\bar{a}_t$ and by $I_t^b$ the component of $C_t\setminus \{a_t,b_t\}$ that contains the point $\bar{b}_t$.  

The regular part $C_0^*=C_0\setminus \{p_0\}$ consists of two components $C_0^a$ and $C_0^b$. We set
  \begin{equation}
 \frac{1}{4\pi}\omega_0^a:=\frac{1}{4\pi}\int_{C_0^a} h_0 ds,\;\;\omega_0^b:=\frac{1}{4\pi}\int_{C_0^b}h_0ds,\;\;\omega_0:=\frac{1}{4\pi}\int_{C_0^*}h_0ds=\omega_0^a+\omega_0^b.
  \label{eq: mc0}
  \end{equation}
  Note  that the limits $\lim_{t\ra 0} L_t^a$,   $ \lim_{t\ra 0} L_t^b$ exist and are finite. We denote them by $L_0^a$ and respectively $L_0^b$. We have
  \[
  L_0^a+L_0^b=L_0:={\rm length}\,(C)_0.
  \]
 Let $D_t$ denote the  restriction of $\bpar$ to the  cooriented  curve $C_t$, $t\neq 0$. As explained in the previous  section we have
 \[
D_t =\begin{cases}
-\ii \frac{d}{ds}+\frac{1}{2}h_t, & t>0,\\
 (-\ii \frac{d}{ds}+\frac{1}{2}h_t)|_{C_t^a}\oplus (-\ii \frac{d}{ds}+\frac{1}{2}h_t)|_{C_t^b}& t<0.
 \end{cases}
 \] 
 If we set 
\[
\rho_t=\omega_t-\lfloor \omega_t\rfloor,\;\; \rho^a_t=\omega_t-\lfloor \omega_t^a\rfloor,\;\;\rho^b_t=\omega_t-\lfloor \omega_t^b\rfloor,
\]
then the computations in Section \ref{s: 1} imply
\begin{equation}
\xi(t):=\xi_{D_t}=\frac{1}{2}\begin{cases}
1-2\rho_t, & t>0\\
(1-2\rho^a_t)+(1-2\rho^b_t), & t<0.
\end{cases}
\label{eq: xit}
\end{equation}

\medskip

\noindent \ding{43}  \emph{Throughout this and the next section we assume that both $D_{\pm 1} $ and are invertible.}

\medskip

We organize  the   family  of  complex  Hilbert spaces $L^2(C_t,ds;\bC)$, $t\in [-1,1]$ as a trivial bundle of Hilbert spaces as follows. 

First observe that  $C_0\setminus \{\bar{a}_0,\bar{b}_0,p_0\}$ is a disjoint union of four open arcs $I_1,\dotsc, I_4$ labeled as in Figure \ref{fig: cobdeg2}. Denote by $\ell_j$ the length of $I_j$ so that
\[
L_0=\ell_1+\cdots +\ell_4,\;\;L_0^a=\ell_1+\ell_4,\;\;L_0^b=\ell_2+\ell_3.
\]
For $t>0$ we can  isometrically identify the  oriented open arc $C_t\setminus \bar{a}_t$ with the   open interval $(0,L_t)$. We obtain in this fashion a canonical isomorphism   
\[
\eI^+_t:=L^2(C_t, ds;\bC)\ra L^2(0,L_t;\bC).
\]
 The   rescaling 
\[
(0,L_0)\ra (0,L_t), \;\; (0,L_0)\ni t\mapsto \frac{t}{\lambda_t},\;\;\lambda_t=\frac{L_0}{L_t}
\]
induces a Hilbert space isomorphism 
\[
\eR^+_t:  L^2(0,L_t;\bC)\ra L^2(0,L_0;\bC)=:\bsH_0.
\]
Note that we have a partition $\eP_+$ of $[0,L_0]$
\begin{equation}
0=t_0<t_1<t_2<t_3<t_4=L_0,\;\;t_j-t_{j-1}=\ell_j,\;\;\forall j=1,\dotsc, 4.
\label{eq: p+}
\end{equation}
This defines a Hilbert space isomorphism 
\[
\eU_+: L^2(0,L_0;\bC)\ra \bigoplus_{j=1}^4 L^2(t_{j-1},t_j;\bC)=\bigoplus_{j=1}^4 L^2(I_j,ds;\bC)=:\bsH_0.
\]
For  $t<0$ we have
\[
L^2(C_t,ds;\bC)=L^2(C_t^a,ds;\bC)\oplus  L^2(C_t^b,ds;\bC).
\]
By removing the points $\bar{a}_t$ and $\bar{b}_t$ we obtain Hilbert space isomorphisms 
\[
L^2(C_t^a,ds;\bC)\ra L^2(0,L_t^a;\bC),\;\;  L^2(C_t^b,ds;\bC)\ra L^2(0,L_t^b;\bC)
\]
that add up to a Hilbert space isomorphism
\[
\eI_t^-:L^2(C_t,ds;\bC)\ra L^2(0,L_t^a;\bC)\oplus L^2(0,L_t^b;\bC).
\]
By rescaling we obtain a Hilbert space isomorphism
\[
\eR^-_t: L^2(0,L_t^a;\bC)\oplus L^2(0,L_t^b;\bC)\ra L^2(0,L_0^a;\bC)\oplus L^2(0,L_0^b;\bC)\cong L^2(0, L_0;\bC).
\]
Next observe that we have   isomorphisms
\[
\eU_-^a:L^2(0,L_0^a;\bC)\ra L^2(I_1,ds;\bC)\oplus L^2(I_4,ds\bC),
\]
\[
\eU_-^b: L^2(0,L_0^b;\bC)\cong L^2(I_3,ds;\bC)\oplus L^2(I_3,ds;\bC),
\]
that add up to an isomorphisms
\[
\eU_-: L^2(0,L_0;\bC)\ra \bigoplus_{j=1}^4 L^2(I_j,ds;\bC).
\]
For $t=0$  we let $\eJ_0$ be the  natural isomorphism
\[
\eJ_0:L^2(C_0,ds;\bC)\ra \bigoplus_{j=1}^4 L^2(I_j,ds;\bC)\cong \bsH_0.
\]
Now define
\[
\eJ_t:=\begin{cases}
\eU_+\eR_t^+\eI_t^+, & t>0,\\
\eU_-\eR_t^-\eI_t^-, &t<0,\\
\eJ_0,& t=0.
\end{cases}
\]
We use the collection of isomorphisms $\eJ_t$ organizes the collection $L^2(C_t,ds;\bC)$ as a trivial Hilbert $\eH$ bundle over $[-1,1]$.

\begin{remark} Let us observe  that any continous function $f:\Sigma\ra \bC$ induces  elements $f|_{C_t}\in L^2(C_t,ds;\bC)$, $\forall t\in [-1,1]$ which   in turn define a continuous section  of the trivial Hilbert  bundle $\eH$.\qed
\end{remark}

\begin{theorem}   (a) The  operators $\eD_t:=\eJ_tD_t\eJ_t^{-1}$ converge in the gap topology as $t\ra 0^\pm$ to  Fredholm, selfadjoint operators  $\eD_0^\pm$.

\noindent (b) The  eta invariants of $\eD_0^\pm$ exist, and  we set
\[
\xi_\pm :=\frac{1}{2}\left(\dim\ker \eD_0^\pm+\eta_{\eD_0^\pm}(0)\,\right),
\]
If $\ker \eD_0^\pm =0$ then we have\footnote{The condition $\ker \eD_0^\pm=0$ is satisfied for  an open and dense set of metrics $g$ satisfying (\ref{eq: can}).  When this condition is  violated       the identity (\ref{eq: main}) needs to be  slightly modified  to take into account these kernels.}
\begin{equation}
i_{APS}(\bpar)+ \lim_{\ve\ra 0^+}SF\bigl(\,\eD_t; \ve<t \leq 1\,\bigr)+\lim_{\ve \ra 0^+}SF\bigl(\,\eD_t,\; -1\leq t < -\ve\,\bigr)=-(\xi_+-\xi_-).
\label{eq: main}
\end{equation}
\label{th: main}
\end{theorem}

\begin{proof}  We set 
\[
\eS_t:= \begin{cases}
 \eU_{+}^{-1} \eD_t \eU_+,& t>0\\
\eU_{-}^{-1} \eD_t \eU_-,& t<0.
\end{cases}
\]
To establish the convergence  statements we  show   that  the  limits  $\lim_{t\ra 0^\pm}\eS_t$ exist in the  gap topology of  the space of  unbounded selfadjoint operators on $L^2(0,L_0;\bC)$. We discuss separately the cases $\pm t>0$, corresponding to restrictions to level sets above/below the critical level set $\{f=0\}$.

\medskip

\noindent {\bf A.} $t>0$.  We observe that
 \[
 \Dom(\eS_t)= \Bigl\{ u \in L^{1,2}(0,L_0;\bC);\;\;u(L_0)= u(0)\,\Bigl\},\;\;\eS_t(u)=-\ii \lambda_t\frac{d}{ds} + \frac{1}{2} h_t\bigl(\,s/\lambda_t\,\bigr),
 \]
 where  we recall that the constant $\lambda_t$ is  the rescaling factor $L_0/L_t$. We set
 \[
 A_t(s):=\frac{1}{\lambda_t}\int_0^s  h_t\bigl(\,\si/\lambda_t\,\bigr)d\si
 \]
 Using the fact that $\lambda_t\ra 1$ and Proposition \ref{prop: cont} we see that it suffices to show that $A_t$ is very weakly  convergent in  $\eA_{L_0}$; see Definition \ref{def: a}. Thus it suffices to prove two things.
 
\begin{equation*}
\mbox{The limit $\lim_{t\ra 0^+} A_t(L_0)$ exists.}
\tag{$\mathbf{A_1}$}
\label{tag: a1}
\end{equation*}
 
\begin{equation*}
\mbox{The limits $\lim_{t\ra 0^+}A_t(s)$ exists for almost any $s\in (0,L_0)$.}
\tag{$\mathbf{A_2}$}
\label{tag: a2}
\end{equation*}
  
\noindent {\bf Proof of} (\ref{tag: a1}).  Observe that 
\[
A_t(L_0)=\int_{C_t} h_t ds=\int_{C_t-\eO}h_tds+\int_{\eO\cap C_t} h_t ds, 
\]
where $\eO$ is the  neighborhood  where (\ref{eq: can}) holds. The intersection of  $C_t$ with $\eO$ is depicted in Figure \ref{fig: cobdeg3}.

\begin{figure}[h]
\centering{\includegraphics[height=3in,width=3in]{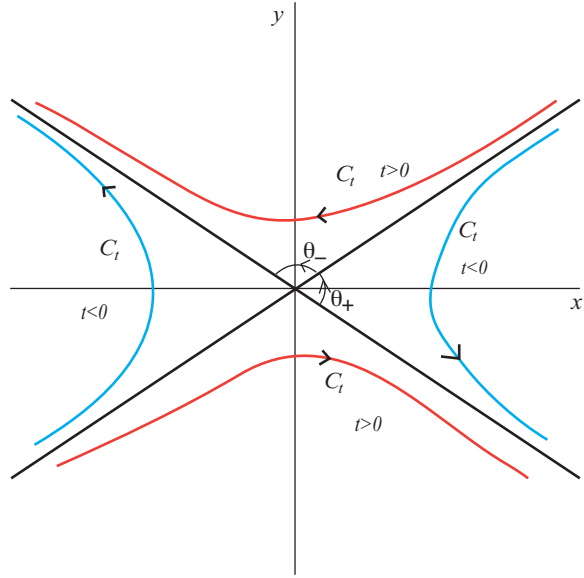}}
\caption{\sl The behavior of $C_t$ near the critical point.}
\label{fig: cobdeg3}
\end{figure}

The integral $\int_{C_t\setminus \eO}h_t ds$ converges as $t\ra 0^+$ to $\int_{C_0\setminus \eO} h_0 ds$.     Next observe that the intersection $C_t\cap \eO$ consists of two oriented arcs (see Figure \ref{fig: cobdeg3}) and  the integral   $\int_{\eO\cap C_t}h_t$ computes  the total angular variation of the   oriented unit  tangent vector field along these oriented arcs.   Using the notations in  Figure \ref{fig: cobdeg3} we see that this total variation approaches $-2\theta_+$ as $t\ra 0+$. Hence
\[
\lim_{t\ra 0^+}  A_t(L_0)=\int_{C_0\setminus 0}h_0 ds -2\theta_+,
\]
so that
 \begin{equation}
 \omega_0^+=\lim_{t\ra 0^+}\omega_t=\frac{1}{4\pi}\lim_{t\ra 0^+} \int_{C_t}h_t ds=\omega_0-\frac{\theta_+}{2\pi}.
 \label{eq: lrmc}
 \end{equation}

\noindent {\bf Proof of} (\ref{tag: a2}).   Let $C_t^*:=C_t\setminus\{\bar{a}_t\}$ and define  $s=s(q): C_t^*\ra (0,\infty)$  to  be the coordinate function  on $C_t^*$ such that the resulting map 
\[
C_t^*\ra \bR, \;\;q\mapsto \si(q) =s(q)/\lambda_t
\]
is an orientation preserving isometry onto $(0,L_t)$.   In other words $\si$ is the oriented arclength function  measured starting at $\bar{a}_t$, and $s$ defines a diffeomorphism $C_t^*\ra (0,L_0)$. Let $q_t: (0,L_0)\ra C_t^*$ be the inverse of this diffeomorphism.

Consider the partition (\ref{eq: p+}).  Observe that there exists  positive constants $c$ and $\ve$ such that whenever 
 \[
 \forall t\in (0,\ve),\;\;\forall s\in [t_1-c,t_1+c]\cup [t_3- c,t_3+c]:\;\; q_t(s)\in \eO
 \]
  the numbers  $t_j$ are defined by (\ref{eq: p+}).   Intuitively the intervals  $[t_1-c,t_1+c]\cup [t_3- c,t_3+c]$ collect the parts of $C_t$  that are close to the critical point $p_0$.   The length of each of the  two components  of $C_t$ that are close to $p_0$ is   bounded  from below  by $2c/\lambda_t$.
  
   To prove part (b) it suffices to understand the behavior of $A_t(s)$ for $s\in [t_1-c,t_1+c]\cup [t_3- c,t_3+c]$. We do this for one of the components since the behavior   for the other component is entirely similar.   We look at the component of $C_t\cap \eO$ that lies in the lower half-plane in Figure \ref{fig: cobdeg3}).  
      
   Here is a geometric approach.  As explained before the difference  $A_t(s)-A_t(t_1-c)$  computes the      angular variation of the unit tangent over the interval $[t_1-c,s]$.  A close look  at Figure \ref{fig: cobdeg3} shows that the absolute value of this is bounded  above by $\theta_+$.    This proves the boundedness part of the bounded convergence.  The almost everywhere convergence is  also obvious in view of the above   geometric interpretation.  The limit function is a bounded function $ A_0:[0,L_0]\ra \bR$ that has  jumps $-\theta_+$ at $t_1$ and  $t_3$
   \[
   A_0(t_1^+)-A_0(t_1^-)=A(t_3^+)-A(t_3^-)=-\theta_+,
   \]
   while the     continuous function
   \[
   A_0(t)+\theta_+H(t-t_1)+\theta_+H(t-t_3)
   \]
   is   differentiable   everywhere on $[0,L_0]\setminus \{t_1,t_3\}$ and the derivative  is the  mean curvature function $h_0$ of $C_0\setminus \{p_0\}$.

  We can now invoke  Theorem \ref{th: split} to conclude that   the  operators $\eD_t$ converge as  $t\ra 0^+$  to  the operator
  \[
  \eD_0^+:\Dom(\eD_0^+)\subset L^2(0,L_0;\bC)\ra  L^2(0,L_0;\bC),
  \]
  where $\Dom(\eD_0^+)$ consists of functions $ u\in L^2(0,L_0;\bC)$ such that
  \[
  u|_{(t_{j-1},t_j)}\in L^{1,2}(t_{j-1},t_j),\;\; \forall j=1,\dotsc, 4,
  \]
  \[
  u(t_i^+)=e^{\ii \theta_+/2}u(t_i^-),\;\;i=1,3,
  \]
 \[
  u(t_2^-)=u(t_2^+) ,\;\; u(t_4^-)= u(t_0^+),
  \]
  while for $u\in \Dom(\eD_0^+)$ we have
  \[
  \Bigl(\eD_0^+u\,\Bigr)|_{(t_{j-1},t_j)}=\Bigl(\,-\ii \frac{d}{ds}+\frac{1}{2}h_0(s)\,\Bigr)u|_{(t_{j-1},t_j)},\;\;\forall j=1,\dotsc, 4.
  \]
  Using the point of view elaborated in  Remark \ref{rem: bvp}   we  let $I$ denote the disjoint union of the intervals $I_j$, $j=1,\dotsc, 4$. We regard  $\eD_0^+$  as a closed densely defined operator on the Hilbert space $L^2(I,\bC)$ with domain consisting  of quadruples $\bu=(u_1,\dotsc ,u_4)\in L^{1,2}(I)$ satisfying the boundary condition
  \[
  \pa_-\bu = \bsT_+ \pa_+ \bu,
  \]
  where $\pa_\pm$ denotes the restriction to the outgoing/incoming boundary component of $I$, while  
  \[
  \bsT_+: \bC^4\cong L^2(\pa_+ I)\ra L^2(\pa_+I)\cong \bC^4,
  \]
  is the transmission operator given by the  unitary $4\times 4$ matrix
  \[
  \bsT_+=\left[
  \begin{array}{cccc}
  0 & 0 & 0 &1\\
  e^{\ii\theta_+/2}& 0 & 0 &0\\
  0 & 1 & 0& 0\\
  0 &0 &e^{\ii\theta_+/2} & 0
  \end{array}
  \right],\;\;\mbox{and} \;\; \eD_0^+\left[
  \begin{array}{c}
  u_1\\
  \vdots\\
  u_4
  \end{array}\right]= \left(\,-\ii\frac{d}{ds}+\frac{1}{2}h_0\,\right) \left[
  \begin{array}{c}
  u_1\\
  \vdots\\
  u_4
  \end{array}\right].
  \]
 Using (\ref{eq: xi}) we deduce that
 \begin{equation}
 \xi^+=\xi_{\eD_0^+}= \frac{1}{2}(1-2\rho_+),\;\; \rho_+=\omega_0^+-\lfloor \omega_0^+\rfloor=\omega_0-\frac{\theta_+}{2\pi}-\left\lfloor \omega_0-\frac{\theta_+}{2\pi}\right\rfloor.
 \label{eq: xi+}
 \end{equation}
 
 \medskip
 
 \noindent {\bf B.} $t<0$.   We observe that $\eS_t=\eS_t^a\oplus \eS_t^b$, where for $\bullet=a,b$ we have
 \[
 \eS_t^\bullet:\Dom(\eS_t^\bullet)\subset L^2(0,L_0^\bullet;\bC)\ra  L^2(0,L_0^\bullet;\bC),
 \]
 \[
 \Dom(\eS_t^\bullet)=\bigl\{ u\in L^{1,2}(0,L_0^\bullet;\bC);\;\; u(L_0^\bullet)=u(0)\,\bigr\},\;\;\eS_t^\bullet u=  -\ii \lambda_t^\bullet \frac{d}{ds} + \frac{1}{2} h_t\bigl(\,s/\lambda_t^\bullet\,\bigr),
 \]
 and $\lambda_t^\bullet$ is the rescaling factor $\frac{L_0^\bullet}{L_t^\bullet}$.  It is convenient to regard $\eS_t^\bullet$ as defined on the component $C_0^\bullet$ of $C_0^*$.  Observe that $C_0^a\setminus \{\bar{a}_0\}=I_1\cup I_4$ and $C_0^b\setminus\{\bar{b}_0\}= I_2\cup I_3$. Arguing as in the case   $t>0$ we conclude that 
 \begin{equation}
 \lim_{t\nearrow 0} \omega_t^a= \omega_0^a+\frac{\theta_-}{4\pi},\;\;  \lim_{t\nearrow 0} \omega_t^b= \omega_0^a+\frac{\theta_-}{4\pi},\;\;\omega_0^-:=\lim_{t\nearrow 0}\omega_t=\omega_0+\frac{\theta_-}{2\pi},
 \label{eq: lim-}
 \end{equation}
 and that the operators $\eD_t^a$ and $\eD_t^b$ converge  in the gap topology  as $t\ra 0^-$ to operators
 \[
 \eD_0^a:\Dom(\eD_0^a)\subset L^2(I_1)\oplus L^2(I_4)\ra  L^2(I_1)\oplus L^2(I_4),
 \]
 \[
 \eD_0^b:\Dom(\eD_0^b)\subset L^2(I_2)\oplus L^2(I_3)\ra  L^2(I_2)\oplus L^2(I_3),
 \]
 where $ \Dom(\eD_0^a)$ consists of functions $(u_1,u_4)\in L^{1,2}(I_1)\oplus L^{1,2}(I_4)$ such that
 \[
 u_4(\pa_-I_4)=e^{-\ii \theta_-/2}u_1(\pa_+I_1),\;\;u_4(\pa_+ I_4)=u_1(\pa_- I_1),
 \]
 $ \Dom(\eD_0^b)$ consists of functions $(u_2,u_3)\in L^{1,2}(I_3)\oplus L^{1,2}(I_3)$ such that
 \[
 u_2(\pa_-I_2)=e^{-\ii \theta_-/4\pi}u_3(\pa_+I_3),\;\;u_2(\pa_+ I_2)=u_1(\pa_- I_3),
 \]
 where $\theta_-$ is depicted in Figure \ref{fig: cobdeg3},  and
 \[
 \eD_0^a(u_1,u_4)=\bigl(\,-\ii\frac{du_1}{ds}+\frac{1}{2}h_0u_1, -\ii\frac{du_4}{ds}+\frac{1}{2}h_0u_4\,\bigr),
 \]
 \[
 \eD_0^a(u_2,u_3)=\bigl(\,-\ii\frac{du_2}{ds}+\frac{1}{2}h_0u_2, -\ii\frac{du_3}{ds}+\frac{1}{2}h_0u_3\,\bigr).
 \]
The direct sum $\eD_0^-=\eD_0^a\oplus \eD_0^b$   is the closed densely defined linear operator on $L^2(I)$ with domain  of quadruples $\bu=(u_1,\dotsc ,u_4)\in L^{1,2}(I,\bC)$ satisfying the boundary condition
  \[
  \pa_-\bu = \bsT_- \pa_+ \bu,
  \]
  where 
  \[
  \bsT_-:\bC^4\cong L^2(\pa_+ I)\ra L^2(\pa_+ I)\cong \bC^4,
  \]
  is the transmission operator given  by the  unitary $4\times 4$ matrix
  \[
  \bsT_-=\left[
  \begin{array}{cccc}
  0 & 0 & 0 & 1\\
  0& 0 & e^{-\ii\theta_-/2}& 0\\
  0 & 1 & 0& 0\\
  e^{-\ii\theta_-/2} & 0 & 0 & 0
  \end{array}
  \right],\;\;\mbox{and} \;\;\eD_0^-\left[
  \begin{array}{c}
  u_1\\
  \vdots\\
  u_4
  \end{array}\right]= \left(\,-\ii\frac{d}{ds}+\frac{1}{2}h_0\,\right) \left[
  \begin{array}{c}
  u_1\\
  \vdots\\
  u_4
 \end{array} \right].
  \]
  Then
  \[
  \xi_-=\xi_-^a+\xi_-^b,
  \]
  where for $\bullet=a,b$ we have
  \begin{equation}
  \xi_-^\bullet=\frac{1}{2}(1-2\rho_-^\bullet),\;\;\rho_-^\bullet=\omega_0^\bullet+\frac{\theta_-}{4\pi}-\left\lfloor \,\omega_0^\bullet+\frac{\theta_-}{4\pi}\,\right\rfloor.
  \label{eq: xi-}
  \end{equation}
  Combining (\ref{eq: lrmc}) and (\ref{eq: lim-})  with the equality $\theta_++\theta_-=\pi$ we deduce
\begin{equation}
\omega_0^+-\omega_0^-=\lim_{t\searrow 0} \omega_t -\lim_{t\nearrow 0} \omega_{t}=-\frac{1}{2}.
\label{eq: limit}
\end{equation}
To prove  (\ref{eq: main}) we use  the index formula (\ref{eq: ind}). We have 
  \[
  i_{APS}(\bpar) =-\frac{1}{2} -\omega_1+\omega_{-1}  -\xi_{D_1}+\xi_{D_{-1}}.
  \]
  \[
  \stackrel{(\ref{eq: limit})}{=} \omega_0^+-\omega_0^- -\omega_1+\omega_{-1}  -\xi_{D_1}+\xi_{D_{-1}}
  \]
  \[
  =(\omega_0^+ +\xi^+)-(\omega_1+\xi_{D_1}) -(\omega_0^-+\xi^-)+ (\omega_{-1}+\xi_{D_{-1}})-(\xi^+-\xi^-)
  \] 
  \[
  \stackrel{(\ref{eq: sf-xi})}{=} - \lim_{\ve\ra 0^+} SF\bigl(\, \eD_t; \ve <t \leq 1\,\bigr)- \lim_{\ve\ra 0^+}SF\bigl(\,\eD_t,\; -1\leq t < -\ve\,\bigr)-(\xi_+-\xi_-).
  \]

\end{proof} 

\begin{remark}[Twisted Dolbeault operators]  (a)   Here the outline of an   analytic argument proving (\ref{tag: a2}). Using (\ref{eq: can}) we deduce  that this component  has a parametrization compatible with the orientation given by
  \begin{equation}
  y_t=-\left(\zeta_t+mx^2\right)^{1/2},\;\; |x|< d_t
  \label{eq: param}
  \end{equation}
  where $\zeta_t=\frac{t}{\beta}$, $m=\frac{\alpha}{\beta}$ and $d_t$   is such that the length of this arc  is $2c/\lambda_t$.  Observe that there exists $d_*>0$ such that $\lim_{t\ra 0^+}dt=d_*$.  We have
  \[
  dy_t = -mx\left(\zeta_t+mx^2\right)^{-1/2}dx.
  \]
 Set 
 \[
 y_t':=\frac{dy_t}{dx}=-mx\left(\zeta_t+mx^2\right)^{-1/2},
 \]
 \[
 y_t'':=\frac{d^2y_t}{dx^2}=-m\left(\zeta_t+mx^2\right)^{-1/2}+m^2x^2\left(\zeta_t+mx^2\right)^{-3/2}=-\frac{m\zeta_t}{(\zeta_t+mx^2)^{3/2}}.
 \]
 The arclength is
  \[
  d\si^2= \Bigr(1+(y_t')^2\,\Bigl)dx^2=\underbrace{\left(\,1+ \frac{m^2 x^2}{\zeta_t+m x^2}\,\right)}_{=:w(t,x)^2} dx^2. 
  \]
  The   mean curvature  $h_t$ is  found using the Frenet formul{\ae}. More precisely $h_t(x)=\frac{y_t''}{w^3}$. Then
 \[
 h_td\si=h_t wdx=\frac{y_t''}{1+(y_t')^2} dx=-\frac{m\zeta_tdx}{(\zeta_t+mx^2)^{1/2}(\zeta_t+mx^2+m^2x^2)}.
 \]
 We observe now that  we can write $h_t d\si=\phi_t^* (\, \rho_\infty du\, )$, where $\phi_t$ is the rescaling map
 \[
 x\mapsto u=t^{-1/2} x\;\;\mbox{and}\;\; \rho_\infty(u)= -\frac{m\zeta_1}{(\zeta_1+mu^2)^{1/2}(\zeta_1+mu^2+m^2u^2)}.
 \]
 This then allows  us to conclude via a standard argument that  the   densities $h_td\si$ converge very  weakly as $t\ra 0^+$  to a  $\delta$-measure concentrated at the origin.

(b) The  results  in Theorem \ref{th: main} extend without  difficulty to  Dolbeault operators twisted by line bundles.  More precisely,  given a Hermitian   line bundle $L$ and a hermitian connection $A$ on $L$, we can form a Dolbeault operator $\bpar_A: C^\infty(L)\ra C^\infty(L\otimes K^{-1}_\Sigma)$.   Fortunately, all the    line bundles on a   the two-dimensional cobordism   $\Sigma$ are trivializable. We  fix a trivialization so that the connection $A$ can be  identified with a purely imaginary $1$-form 
\[
A= \ii a,\;\;a\in \Omega^1(\Sigma).
\]
Then
\[
\bpar_A=\bpar+\ii a^{0,1}.
\]
The restriction of $D^+_A=\sqrt{2}\bpar_A$ to the cooriented curve $C_t$ is
\[
D_A(t)= -\ii\nabla^A_s+\frac{1}{2}h_t =-\ii\frac{d}{ds}+\frac{1}{2}h_t + a_t,\;\; a_t:=a\Bigl(\, \frac{d}{ds}\,\Bigr)\in \Omega^0(C_t).
\]
  
  As  in the proof of Theorem \ref{th: main}, we only need to understand the behavior of $a_t$ in the neighborhood $\eO\cap C_t$.  Suppose for simplicity $t>0$ and we concentrate only on the component of $C_t\cap \eO$ that lies in the lower half-plane of Figure \ref{fig: cobdeg3}.  In the neighborhood $\eO$ we can write
  \[
  a= pdx+qdy,\;\;p,q\in C^\infty(\eO).
  \]
  Using the parametrization (\ref{eq: param}) we deduce that
  \[
  a|_{C_t\cap \eO}= \Bigl(\,p-mqx(\zeta_t+mx^2)^{-1/2}\,\Bigr) dx= a_t ds=a_twdx
    \]
  Hence, as $t\ra 0^+$, the measure $a_tds$ converges  to the measure $\bigl(\, p- m^{1/2}(2H(x)-1\,)\,\bigr)dx$

  (c) One may ask  what happens in the case of  a cobordism corresponding to a local min/max  of a Morse function. In this case $\Sigma$ is a disk,  the regular level sets  $C_t$ are circles and the singular level set is a point. Consider for example the case of a local minimum. Assume that the metric  near the minimum $p_0$ is Euclidean, and  in some Euclidean coordinates  near $p_0$ we have $f=x^2+y^2$. Then $C_t$ is the  Euclidean circle of radius $t^{1/2}$, and the function $h_t$ is the constant function $h_t=t^{-1/2}$. Then $\omega_t=\frac{1}{2}$, $\xi_t=\frac{1}{2}$ and the   Atiyah-Patodi-Singer index of $\bpar$ on the Euclidean disk of radius $t^{1/2}$ is $0$.   The operator  $D_t$ can be identified with the operator
  \[
  -\ii\frac{d}{ds}+\frac{1}{2t^{1/2}}
  \]
  with periodic boundary conditions on the interval $[0,2\pi t^{1/2}]$. Using the rescaling trick in Remark \ref{rem: resc} we see that this operator is conjugate to the operator $L_t=-t^{1/2}\ii \frac{d}{ds}+\frac{1}{2}$ on the interval  $[0,2\pi]$ with periodic boundary conditions.   The   switched graphs of these operators 
  \[
  \widetilde{\Gamma}_{L_t}=\bigl\{ (L_t u, u);\;\;u\in L^{1,2}(0,2\pi;\,\bC);\;\;u(0)=u(2\pi)\,\bigr\}\subset \bsH\oplus \bsH,\;\;\bsH=L^2(0,2\pi;\,\bC),
  \]
  converge in the gap topology to the subspace $\bsH_+=\bsH\oplus 0\subset \bsH\oplus \bsH$.  This limit is not the switched graph of any    operator. However, this limiting space     forms a Fredholm pair with $\bsH_-=0\oplus \bsH$  and invoking  the results in \cite{Cib}  we conclude that the limit
  \[
  \lim_{\ve \searrow 0} SF( L_t;\;\;\ve \leq t \leq t_0)
  \]
  exists an it is finite.\qed    \end{remark}


\section{The  Kashiwara-Wall index}
\label{s: 4}
In this final section we would like to identify    the correction term    in the right hand side of (\ref{eq: main}) with  a symplectic   invariant that often appears in surgery formul{\ae}.  To this aim, we need to elaborate  on the symplectic point of view  first outlined in Remark \ref{rem: bvp}.

Fix a finite dimensional complex hermitian space  $\bsE$,  let $n:=\dim \bsE$, and set
\[
\widehat{\bsE}:=\bsE\oplus\bsE,\;\;\bsE_+:=\bsE\oplus 0,\;\;\bsE_-:=0\oplus \bsE,
\]
and let $J:\widehat{\bsE}\ra \widehat{\bsE}$  be the unitary operator given by the block decomposition
\[
J=\left[\begin{array}{cc}
-\ii & 0\\
0 &\ii
\end{array}
\right].
\]
We let   $\Lag$ denote the  space of hermitian lagrangians on $\wbe$,  i.e., complex subspaces $L\subset \wbe$ such that $L^\perp=JL$. As explained in \cite{Cib, N1}  any such a lagragian can be identified with the graph\footnote{In \cite{KL} a lagrangian is identified with the graph of  an isometry $\bsE_-\ra \bsE_+$ which explains why our formul{\ae}  will look a bit different than the ones on \cite{KL}. Our  choice is based on the conventions   in \cite{Cib}  which seem to minimize  the number of signs in the Schubert calculus on $\Lag$.} of a complex isometry $T: \bsE_+\ra\bsE_-$, or equivalently, with the    group $U(\bsE)$ of unitary operators on $E$.        In other words, the graph map
\[
\Gamma: U(\bsE) \ra \Lag(\wbe),\;\; U(\bsE)\mapsto \Gamma_T\subset \widehat{\bsE}
\]
is a diffeomorphism. The involution $ L\leftrightarrow JL$ on $\Lag$ corresponds  via this diffeomorphism  to the involution $T\leftrightarrow -T$ on $U(\bsE)$.

We define  a branch of the logarithm  $\log :\bC^*\ra \bC$   by   requiring $\im \log\in [-\pi, \pi)$. Equivalently,
\[
\log z =\int_{\gamma_z} \frac{d\zeta}{\zeta},
\]
where $\gamma_z :[0,1]\ra \bC$ is any smooth path from   $1$ to $z$ such that
\[
\forall t\in[0,1),\;\; \gamma_z(t)\not\in (-\infty, 0].
\]
 In particular, $\log (-1)= \pi\ii$. Following \cite[\S 6]{KL} we define
\[
\tau: U(\bsE)\times U(\bsE)\ra \bR,\;\; \tau(T_0,T_1)=\frac{1}{2\pi\ii}\tr \log(T_1^{-1} T_0)
\]
\[
=\frac{1}{2\pi\ii}\sum_{\lambda \in\bC^*} \bigl(\,\log \lambda\,\bigr)m_\lambda,\;\; m_\lambda:=\dim \ker(\lambda-T_1^{-1}T_0).
\]
 Observe that
  \begin{equation}
   e^{2\pi\ii \tau(T_0,T_1)}=\frac{\det T_0}{\det T_1}
\label{eq: exp-tau}
\end{equation}
Note that
\begin{equation}
\tau(T_0,T_1)+\tau(T_1,T_0)=\dim\ker(T_0+T_1).
\label{eq: flip}
\end{equation}
Via the graph diffeomorphism we obtain a map
\[
\mu=\tau\circ \Gamma :\Lag \times \Lag \ra \bR.
\]
The equality (\ref{eq: flip}) can be rewritten as
\begin{equation}
\tau(L_0,L_1)+\tau(L_1,L_0)=\dim (L_0\cap J L_1)=\dim (JL_0\cap L_1).
\label{eq: flip1}
\end{equation}
We  want to relate the invariant $\tau$ to the eta invariant of a natural    selfadjoint operator. We associate to each pair $L_0,L_1\in \Lag$ the   selfadjoint operator
\[
D_{L_0,L_1}:V(L_0,L_1)\subset L^2(I,\wbe)\ra L^2(I,\wbe),
\]
where
\[
V(L_0,L_1)=\bigl\{ u\in L^{1,2}(I,\wbe);\;\;u(0)\in L_0,\;\;u(1)\in L_1\,\bigr\},\;\;D_{L_0,L_1}u =J\frac{du}{dt}.
\]
This is a   selfadjoint    operator with compact resolvent. We want to describe its spectrum, and in particular, prove that it has a well defined    eta invariant. Let $T_0, T_1: \bsE_+\ra \bsE_-$ denote the isometries   associated to $L_0$ and respectively $T_1$.  Then $T_1^{-1}T_0$  is a unitary operator on $\bsE_+$ so its spectrum consists of complex numbers of norm $1$.

\begin{proposition}   For any $L_0,L_1\in \Lag$ we have
\begin{equation}
\spec D_{L_0, L_1}= \frac{1}{2\ii}\exp^{-1}\Bigl(\, \spec(T_1^{-1}T_0)\,\Bigr).
\label{eq: spec1}
\end{equation}
In particular, the spectrum  of $D_{L_0,L_1}$ consists of finitely many  arithmetic progressions with ratio $\pi$ so that  the eta invariant of $D_{L_0,L_1}$ is well defined.
\label{prop: spec-h}
\end{proposition}

\begin{proof} Observe first  that any $u\in L^2(I,\wbe)$ decomposes as a pair
\[
u=(u_+, u_-),\;\;u_\pm\in L^2(I,\bsE_\pm).
\]
If $u\in V(L_0, L_1)$ is an eigenvector of $D_{L_0,L_1}$ corresponding to an eigenvalue $\lambda$ then  $u$ satisfies  the boundary value problems
\begin{subequations}
\begin{equation}
-\ii\frac{du_+}{dt}=\lambda u_+,\;\;\ii\frac{du}{dt}=\lambda u_-,
\label{eq: eval1}
\end{equation}
\begin{equation}
u_-(0)=T_0u_+(0),\;\;u_-(1)= T_1u_+(1).
\label{eq: eval2}
\end{equation}
\end{subequations}
The equalities (\ref{eq: eval1}) imply that
\[
u_+(1)=e^{\ii \lambda}u_+(0),\;\;u_-(1)=e^{-\ii\lambda}u_-(0).
\]
Using (\ref{eq: eval2}) we deduce
\[
e^{\ii\lambda} T_1 u_+(0)=u_-(1)=e^{-\ii\lambda} u_-(0)= e^{-\ii\lambda} T_0 u_+(0).
\]
Hence
\[
e^{2\ii\lambda} \in \spec (T_1^{-1}T_0)\Longrightarrow   \lambda\in\frac{1}{2\ii}\exp^{-1}\Bigl(\, \spec(T_1^{-1}T_0)\,\Bigr).
\]
Running the above argument  in reverse we deduce that any $\lambda\in \frac{1}{2\ii}\exp^{-1}\Bigl(\, \spec(T_1^{-1}T_0)\,\Bigr)$ is an eigenvalue of $D_{L_0,L_1}$.
\end{proof}

We let $\xi(L_0,L_1)$ denote the reduced  eta invariant of $D_{L_0,L_1}$,
  \[
  \xi(L_0,L_1)=\frac{1}{2}\Bigl(\,\dim\ker D_{L_0,L_1}+ \eta_{D_{L_0,L_1}}(0)\,\Bigr).
  \]
  If $ e^{\ii \theta_1},\dotsc, e^{\ii\theta_n}$, $\theta_1,\dotsc, \theta_n\in (0,2\pi]$,  are the   eigenvalues of $T_1^{-1}T_0$,   then the spectrum of   $D_{L_0,L_1}$ is
\[
\spec\bigl(\, D(L_0,L_1)\,\bigr)=\bigcup_{k=1}^n \Bigl\{ \frac{\theta_k}{2}+\pi\bZ\Bigr\}.
\]
and we deduce as in Section \ref{s: 1} using  (\ref{eq: herm})  that
\[
\eta_{D_{L_0,L_1}}=\sum_{\theta_k\in(0,2\pi)} \Bigl( 1- \frac{\theta_k}{\pi} \,\Bigr),
\]
and
\[
\xi(L_0,L_1)=\frac{1}{2}\sum_{\theta_k\in(0,2\pi)} \Bigl( 1- \frac{\theta_k}{\pi} \,\Bigr)+\frac{1}{2}\dim \ker D_{L_0,L_1}.
\]
On the other hand
\[
\frac{1}{2\pi\ii}\tr\log(- T_1^{-1}T_0)= \frac{1}{2\pi}\sum_{\theta_k\in(0,2\pi]}(\theta_k-\pi)=-\frac{1}{2} \sum_{\theta_k\in(0,2\pi)} \Bigl( 1- \frac{\theta_k}{\pi} \,\Bigr)+\frac{1}{2}\dim \ker(T_0-T_1).
\]
Since $\ker D_{L_0,L_1}\cong \ker (T_0-T_1)\cong L_0\cap L_1$ we conclude
\[
\tau(T_0,-T_1)=\tau(-T_0,T_1)=\tau(JL_0, L_1)= -\xi(L_0,L_1)+ \dim (L_0\cap L_1).
\]
Using (\ref{eq: flip}) we deduce
\begin{equation}
\xi(L_0,L_1)=\tau(T_1,-T_0)=\tau(L_1, JL_0)=\tau(JL_1,L_0).
\label{eq: xi-tau}
\end{equation}

Following \cite{KL}  (see also \cite{CLM}) we associate to each triplet of lagrangians $L_0,L_1,L_2$ the quantity
\[
\omega(L_0,L_1,L_2):=\tau(L_1,L_0)+\tau(L_2, L_1)+\tau(L_0, L_2),
\]
and we will refer to its as the  (\emph{hermitian}) \emph{Kashiwara-Wall index}  (or simply the \emph{index}) of the triplet. Observe that $\omega$ is indeed an integer since (\ref{eq: exp-tau}) implies that
\[
e^{2\pi\ii \omega(L_0,L_1,L_2)}= 1.
\]
We set
\[
d(L_0,L_1,L_2):=\dim (JL_0\cap L_1)+\dim(JL_1\cap L_1)+\dim(JL_2\cap L_0).
\]
Using (\ref{eq: flip1}) we deduce that for any permutation $\vfi$ of $\{0,1,2\}$ with signature $\eps(\vfi)\in\{\pm 1\}$ we have
\begin{equation}
\omega(L_0,L_1,L_2)-\eps(\vfi)\omega(L_{\vfi(0)},L_{\vfi(1)}, L_{\vfi(2)})=d(L_0,L_1,L_2)\times \begin{cases}
0, &\mbox{$\vfi$ even}\\
1, & \mbox{$\vfi$ odd}.
\end{cases}
\label{eq: skew1}
\end{equation}
We want to apply the above facts to a special choice of $\wbe$.      Let $\bsI$ denote the disjoint union  of the intervals $I_1,\dotsc, I_4$ introduced in Section \ref{s: 3}.  They were obtained  by removing the points $\bar{a}_0$, $p_0$ and $\bar{b}_0$ from the critical level set $C_0$;   Figure \ref{fig: cobdeg2}.  We interpret $\bsI$ as an oriented $1$-dimensional with boundary and we let
\[
\wbe:= L^2(\pa \bsI),\;\; \bsE_\pm =L^2(\pa_\pm \bsI).
\]
The    spaces $\bsE_\pm$  have canonical bases and thus we can identify both of them with the standard  Hermitian space  $\bsE=\bC^4$. Define $J:\wbe\ra \wbe$ as before. We have  a canonical differential operator
\[
D_0: C^\infty(\bsI,\bC)\ra C^\infty(\bsI,\bC),\;\;D_0 \left[
\begin{array}{c}
u_1\\
\vdots\\
u_4
\end{array}
\right]=\left[
\begin{array}{c}
-\ii\frac{du_1}{dt}+\frac{1}{2}h_0|_{I_1}\\
\vdots\\
\vdots\\
-\ii\frac{du_1}{dt}+\frac{1}{2}h_0|_{I_4}
\end{array}
\right],
\]
We set
\[
\omega_k:=\frac{1}{4\pi}\int_{I_k}h_0 ds
\]
so that
\[
\omega_0=\omega_1+\cdots+\omega_4,\;\;\omega_0^a=\omega_1+\omega_4,\;\;\omega_0^b=\omega_2+\omega_3.
\]
We have a natural restriction map
\[
r:C^\infty(\bsI,\bC)\ra L^2(\pa \bsI,\bC) =\wbe
\]
 and we define the \emph{Cauchy data space}   of $D_0$ to be the subspace
 \[
 \Lambda_0:= r(\ker D_0)\subset \wbe.
 \]
 We can verify easily that  $\Lambda_0$ is a  Lagrangian subspace of  $\wbe$ that is described by the isometry $\bsT_0:\bsE_+\ra \bsE_-$ given by the diagonal matrix
 \[
 \bsT_0=\diag\bigl(\,e^{2\pi\ii\omega_1},\dotsc, e^{2\pi\ii\omega_4}\,\bigr).
 \]

\medskip

\noindent \ding{43}  \emph{In the remainder of this section we assume\footnote{This assumption is  satisfied for a generic choice of metric on $\Sigma$.}     that the operators $\eD_0^\pm$ that appear in Theorem \ref{th: main} are invertible.}

\begin{proposition}  Let $\eD_0^\pm$ be the operators that appear in Theorem \ref{th: main}. Then
 \begin{equation}
 \xi_{\eD_0^\pm}= -\xi\bigl(\,\Gamma_{\bsT_\pm},\Lambda_0\,\bigr)=\xi\bigl(\,\Lambda_0,\Gamma_{\bsT_\pm}\,\bigr)=\tau(\Gamma_{\bsT_\pm}, J\Lambda_0)
 \label{eq: eta-kashi}
 \end{equation}
 \label{prop: eta-kashi}
 \end{proposition}

 \begin{proof}   We need to  find the spectra of  $\bsT_0^{-1}\bsT_\pm $. We set $z_k =e^{-2\pi\ii\omega_k}$, $k=1,\dotsc,4$ and  $\rho=e^{-\ii\theta_+/2}$, so that $e^{\-\ii\theta_-/2}=-\ii \rho$. Then
 \[
 \bsT_0^*\bsT_+= \left[
  \begin{array}{cccc}
  0 & 0 & 0 &z_1\\
  z_2\rho& 0 & 0 &0\\
  0 & z_3 & 0& 0\\
  0 &0 &z_4\rho & 0
  \end{array}
  \right],\;\;\bsT_0^*\bsT_-= \left[
  \begin{array}{cccc}
  0 & 0 & 0 & z_1\\
  0& 0 & -\ii z_2\rho& 0\\
  0 & z_3 & 0& 0\\
  -\ii z_4 \rho & 0 & 0 & 0
  \end{array}
  \right].
  \]
  The eigenvalues of $\bsT_0^*\bsT_+$ are the fourth order roots of
  \[
  \zeta= \rho^2z_1\cdots z_4= e^{\ii (\theta_+ -2\pi\omega_0)}.
  \]
  Hence
  \[
  \exp^{-1} \bigl(\,\spec(\bsT_0^*\bsT_+)\,\bigr)=  \frac{\ii (\theta_+-2\pi\omega_0)}{4}+\frac{\pi\ii}{2}\bZ
  \]
  and  using (\ref{eq: spec1}) we deduce
  \[
  \spec \bigl(\, D_{\Gamma_{\bsT_+},\Lambda_0}\,\bigr)= \frac{\pi}{2}\Bigl\{\,\Bigl( \frac{\theta_+}{2\pi}-\omega_0\,\Bigr)+\bZ\,\Bigr\}.
  \]
 The eigenvalues of $\bsT_0^*\bsT_-$ are the square roots of
 \[
 z_1z_4e^{\ii \theta_-/2}=e^{-\ii (\theta_-/2+2\pi\omega_0^a)}\;\; \mbox{and}\;\; z_1z_4e^{\ii \theta_-/2}-e^{-\ii (\theta_-/2+2\pi\omega_0^b)}.
 \]
 Hence
  \[
\spec \bigl(\, D_{\Gamma_{\bsT_+},\Lambda_0}\,\bigr)=\left\{ -\pi\left(\frac{\theta_-}{4\pi}+\omega_0^a\right)+\pi\bZ\,\right\}\cup\, \left\{ -\pi\left(\,\frac{\theta_-}{4\pi}+\omega_0^b\,\right)+\pi\bZ\,\right\}.
  \]
  The desired conclusion follows using  (\ref{eq: xi+}), (\ref{eq: xi-})  and (\ref{eq: herm}). \end{proof}

   \begin{theorem}  Under the same assumptions and notations as in Theorem \ref{th: main} we have
   \[
   i_{APS}(\bpar)+ \lim_{\ve\ra 0^+}SF\bigl(\eD_t; \ve<t \leq 1\bigr)+\lim_{\ve \ra 0^+}SF\bigl(\eD_t,\; -1\leq t < -\ve\bigr)=-\omega( J\Lambda_0, \Gamma_{\bsT_+},\Gamma_{\bsT_-}) .
   \]
   \label{th: kashi}
   \end{theorem}

   \begin{proof}     We have
   \[
    i_{APS}(\bpar)+ \lim_{\ve\ra 0^+}SF\bigl(\eD_t; \ve<t \leq 1\bigr)+\lim_{\ve \ra 0^+}SF\bigl(\eD_t,\; -1\leq t < -\ve\bigr)\stackrel{(\ref{eq: main})}{=}-(\xi_+-\xi_-)
\]
\[
     \stackrel{(\ref{eq: eta-kashi})}{=}-\tau(\Gamma_{\bsT_+},\Lambda_0)-\tau(J\Lambda_0,\Gamma_{\bsT_-})=-\omega(J\Lambda_0, \Gamma_{\bsT_+},\Gamma_{\bsT_-})+\tau(\Gamma_{\bsT_+},\Gamma_{\bsT_-}).
\]
 To compute  $\tau(\Gamma_{\bsT_+},\Gamma_{\bsT_-})=\tau(\bsT_+,\bsT_-)$ we  need to compute the spectrum of    $\bsT_-^*\bsT_+$. We set $\rho=e^{\ii\theta_+/2}$ so that $e^{-\ii\theta_-/2}=-\ii \rho$. We have
\[
\bsT_-^*\bsT_+= \left[
  \begin{array}{cccc}
  0 & 0 & 0 & \ii\bar{\rho}\\
  0& 0 & 1& 0\\
  0 & \ii\bar{\rho} & 0& 0\\
  1 & 0 & 0 & 0
  \end{array}
  \right]\,\cdot \,  \left[
  \begin{array}{cccc}
  0 & 0 & 0 &1\\
  \rho& 0 & 0 &0\\
  0 & 1 & 0& 0\\
  0 &0 &\rho & 0
  \end{array}
  \right]=\left[
  \begin{array}{cccc}
  0 & 0 & \ii & 0\\
  0 & 1 &  0 & 0\\
  \ii & 0 & 0 & 0\\
  0 & 0 & 0 & 1
  \end{array}
  \right]
  \]
  From the second and forth column we see that $1$ is an eigenvalue of $\bsT_-^*\bsT_+$ with multiplicity $2$.  The other two eigenvalues are  $\pm \ii$, namely  the eigenvalues of the $2\times 2$ minor
  \[
  \left[
  \begin{array}{cc}
  0 &\ii\\
  \ii & 0
  \end{array}
  \right].
  \]
 This shows that   $\tau(\bsT_+,\bsT_-)=0$.  \end{proof}

Let $I$ denote the  unit interval $[0,1]$. We set $\pa_+ I=\{1\}$, $\pa_- I=\{0\}$.   We identify $\widehat{\bsE}$ with the  finite dimensional  Hilbert space $L^2(\pa I, \bsE)$ and the Hilbert spaces $\bsE_\pm$  with


\end{document}